\documentclass[oneside,english]{amsart}
\usepackage{amsfonts}
\usepackage{graphicx}
\usepackage{amsmath,amssymb,amsfonts,amsthm,mathrsfs}
\usepackage[english]{babel}
\usepackage{hyperref}
\usepackage{marginnote}
\usepackage{color}
\usepackage{enumerate}
\usepackage[numbers]{natbib}

\hypersetup{
  colorlinks=true,
  linkcolor=blue,    % 公式、定理和章节引用
  citecolor=red,     % 参考文献引用
  urlcolor=magenta   % 网页链接
}
\usepackage[bottom]{footmisc}

\newcommand{\R}{\mathbb{R}}

\theoremstyle{plain}

\newtheorem{theorem}{Theorem}[section]
\newtheorem{lemma}[theorem]{Lemma}
\newtheorem{proposition}[theorem]{Proposition}

\theoremstyle{definition}

\theoremstyle{remark}
\newtheorem{remark}[theorem]{Remark}
\numberwithin{equation}{section}

\makeatletter
\def\subsection{\@startsection{subsection}{2}%
  \z@{.5\linespacing\@plus.7\linespacing}{-.5em}%
  {\normalfont}}
\makeatother

\usepackage{tikz}

\newcommand{\N}{\mathbb{N}}

\newcommand{\la}{\lambda}

\begin{document}
\pagestyle{plain}
\title{Compactness of positive radial Solution Sets and corresponding \(L^2\)-Mass sets for ``Zero Mass''
Quasi-linear Schr\"odinger Equations}

\author{Haidong Liu}
\address{Institute of Mathematics, Jiaxing University, Zhejiang 314001, China}
\email{liuhaidong@zjxu.edu.cn}

\author{Yulan Tang}
\address{School of Mathematical Sciences, Capital Normal University, Beijing 100048, China}
\email{tangyulan@cnu.edu.cn}

\author{Chengcheng Wu}
\address{School of Mathematical Sciences, Shanxi University, Taiyuan 030006, China}
\email{wuchengcheng@amss.ac.cn}

\begin{abstract}
We study the compactness of two sets of positive radial solutions to ``zero mass'' quasi-linear Schr\"odinger equation
\[
-\Delta u-u\Delta(|u|^2)=g(u)
\quad\text{in }\mathbb R^N,\quad N\ge5.
\]
For nonlinearities that are either strictly subcritical or
asymptotically critical at infinity, we prove that
the set of least energy positive radial solutions is nonempty and compact in the natural space. Moreover, in the strictly subcritical case and under
additional assumptions, we show that the set of all finite $L^2$-mass positive radial solutions is compact in \(L^2(\mathbb R^N)\). 
The proof relies on uniform decay estimates for the corresponding positive
radial solutions of the transformed semilinear equation, which yield the
required uniform \(L^2\)-tail control.

\begin{flushleft}
\end{flushleft}
\end{abstract}

\date{}
\maketitle

\thispagestyle{empty}

\section{Introduction}

In this paper, we are concerned with the compactness of positive radial solution sets and of their
\(L^2\)-mass sets to the quasi-linear Schr\"odinger equation
\begin{equation}\label{eq:1.1}
-\Delta u-u\Delta(|u|^2)=g(u)
\quad\text{in }\mathbb R^N,\quad N\ge 5.
\end{equation}

The following time-dependent quasi-linear equation 
\begin{equation}\label{eq:intro_time_dependent}
i\partial_t\phi +\Delta \phi+\Delta(|\phi|^2)\phi+F(|\phi|^2)\phi=0, \quad (t,x)\in\R \times \R^N,
\end{equation}
appears in dissipative quantum mechanics, in plasma physics and in fluid mechanics; see \cite{MR870992,kurihara1981large,MR727767}. Here $\phi: \R \times \R^N\to \mathbb{C}$ is a complex wave function and $F: \R \to \R$ is a real-valued function. A standing wave of \eqref{eq:intro_time_dependent} is a solution of the form $\phi=e^{i\kappa t}u(x)$, where $\kappa \in \R$ is a constant and \(u\) is a time-independent function. Thus, solutions to \eqref{eq:1.1} correspond to standing waves of \eqref{eq:intro_time_dependent} with $\kappa=0$ and $g(s)=F(s^2)s$. \eqref{eq:1.1} is called the ``zero mass'' quasi-linear Schr\"odinger equation, roughly speaking, if $g(0)=g^{\prime}(0)=0$. 

We recall the semilinear ``zero mass'' equation
\begin{equation}\label{eq:intro_semilinear_zero_mass}
-\Delta v=f(v)\quad\text{in }\mathbb R^N, \quad N\ge 3.
\end{equation}
For \(f(v)=v^{p-1}\), the classical theorem of Gidas and Spruck
\cite{MR615628} implies that \eqref{eq:intro_semilinear_zero_mass} has no
positive classical solution for \(2<p<2^*\). In contrast, if \(p\ge 2^*\), then
\eqref{eq:intro_semilinear_zero_mass} admits a unique positive radial
classical solution up to scaling; see \cite{MR1255211}. For general nonlinearities,
Berestycki and Lions \cite{MR695535} used a constrained minimization method to
prove that \eqref{eq:intro_semilinear_zero_mass} admits a positive,
spherically symmetric, radially decreasing classical solution in
\(D^{1,2}(\mathbb R^N)\). Flucher and M\"uller \cite{MR1617704} proved that a
ground state \(v_0\) satisfies the polynomial decay \(v_0(x)\le C|x|^{2-N}\) for
\(|x|\ge1\). $v_0$ belongs to \(L^2(\mathbb R^N)\) exactly in dimensions
\(N\ge5\). Azzollini and Pomponio \cite{MR2450560} used variational methods to
establish existence and multiplicity results. Jeanjean et al. \cite[Theorem~2.2 $(ii)$]{MR4701352} showed a Liouville-type nonexistence result for \eqref{eq:intro_semilinear_zero_mass}.

Limited work has been done in
\eqref{eq:1.1}. For the pure power equation
\begin{equation}\label{eq:intro_pure_power_qls}
-\Delta u-u\Delta(|u|^2)=|u|^{p-2}u
\quad\text{in }\mathbb R^N, \quad N\ge3,
\end{equation}
if \(2^*<p<2\cdot2^*\), Adachi and Watanabe
\cite{MR3427691} established the uniqueness and nondegeneracy of the positive
solution of \eqref{eq:intro_pure_power_qls} by applying the ODE method
developed in \cite{MR1843289}. Cheng and Wei \cite{MR3989948} showed that \eqref{eq:intro_pure_power_qls} has no fast-decaying solutions if \(2<p\le 2^*\) or \(p\ge 2\cdot 2^*\) and admits a one-parameter family of
slow-decaying positive radial solutions if \(p>2^*\). 

Several recent works showed that ``zero mass'' problems are closely related to the
existence and nonexistence of normalized solutions. We first consider normalized
ground state solutions of
\[
\left\{
\begin{array}{ll}
-\Delta v+\kappa v=f(v), \quad
& \text{in }\mathbb R^N,\quad N\ge3,\\[1mm]
\displaystyle\int_{\mathbb R^N}v^2\,dx=a,
&
\end{array}
\right.
\]
Bieganowski and Mederski \cite{MR4232669} minimized the energy functional on a Pohozaev--Nehari manifold intersected with the closed ball in $L^2(\mathbb{R}^N)$ and then proved that the
minimizer actually belongs to the boundary of the closed ball. Liu and Zhao \cite{MR4660103} showed that the argument in
\cite{MR4232669} relies on the nonexistence, under suitable assumptions, of
nontrivial solutions to \eqref{eq:intro_semilinear_zero_mass}.
For the following quasi-linear equation
\begin{equation}\label{eq:intro_normalized_qls}
\left\{
\begin{array}{ll}
-\Delta u-\Delta(u^2)u+\kappa u=|u|^{p-2}u,
&  \quad \text{in }  \mathbb R^N,\\[1mm]
\displaystyle\int_{\mathbb R^N}u^2\,dx=a,
&
\end{array}
\right.
\end{equation}
the \(L^2\)-mass of positive solution to \eqref{eq:intro_pure_power_qls} determines the threshold in the prescribed
mass \(a\) that separates the existence and nonexistence of normalized ground states for \eqref{eq:intro_normalized_qls}. If \(1\le N\le2\) and
\(2<p<\infty\), \eqref{eq:intro_pure_power_qls} has only trivial solution. If \(N=3,4\) and \(p>2^*\) \eqref{eq:intro_pure_power_qls} has a unique positive \(D^{1,2}\) solution, but is not in \(L^2(\mathbb R^N)\). Jeanjean et al.\cite{jeanjean2025existence} obtained that \eqref{eq:intro_normalized_qls} possesses normalized ground states for any \(a>0\) if $1\le N\le 4$ and $4+\frac{4}{N}< p<2\cdot 2^*$. For \(N\ge5\) and \(4+\frac4N<p<2\cdot 2^*\), \eqref{eq:intro_pure_power_qls} 
admits a unique positive radial solution \(u_0\) with finite \(L^2\)-mass. If
\(a_0=\|u_0\|_2^2\), then normalized ground states are attained precisely
for \(0<a\le a_0\), whereas the ground state is not attained for
\(a>a_0\). 

For positive normalized solutions, the ``zero mass''
equation also plays a important role in the global branch analysis. Consider
\begin{equation}\label{eq:intro_global_branch_semilinear}
-\Delta v+\kappa v=f(v)\quad\text{in }\mathbb R^N,\quad \kappa>0.
\end{equation}
In \cite{MR4701352}, the nonexistence of positive radial classical solutions to \eqref{eq:intro_semilinear_zero_mass} was used to prove the convergence, after a suitable rescaling, of positive solutions \(u_\kappa\) of \eqref{eq:intro_global_branch_semilinear} as \(\kappa\to0^+\) and establish the uniqueness of positive solutions of \eqref{eq:intro_global_branch_semilinear} for sufficiently small \(\kappa>0\). In view of the preceding observations, one of the main motivations for the paper is to advance the study of 
normalized solutions for quasi-linear Schr\"odinger equations by establishing
compactness of the \(L^2\)-mass set associated with positive radial solutions to
\eqref{eq:1.1}.

The energy functional associated with \eqref{eq:1.1} is 
\[
I(u)=\frac12\int_{\R^N}|\nabla u|^2\,dx
+\int_{\R^N}u^2|\nabla u|^2\,dx
-\int_{\R^N}G(u)\,dx,
\]
defined on the natural space
\[X=
\left\{
 u\in H^1(\mathbb R^N):
 \int_{\mathbb R^N}u^2|\nabla u|^2\,dx<\infty
\right\},
\]
where $G(t)=\int_0^t g(s)\,ds$.
$u_0$ is a weak solution of \eqref{eq:1.1} if and only if $$ \lim\limits_{t\rightarrow 0^+}\frac{I(u_0+t\Phi)-I(u_0)}{ t } =0$$
for any $\Phi \in C_0^{\infty}(\mathbb{R}^N,
\mathbb{R})$. Although \(X\) is not a vector space, it is a complete metric space
with respect to
\[
d_X(u_1,u_2)=
\|u_1-u_2\|_{H^1(\mathbb R^N)}
+
\|\nabla(u_1^2)-\nabla(u_2^2)\|_{L^2(\mathbb R^N)};
\]
see \cite[Section~2]{MR2630099}. A positive solution \(u_0\) of
\eqref{eq:1.1} is called a least energy positive solution if it minimizes the
energy functional \(I\) among all positive solutions of \eqref{eq:1.1}. 
We consider the solution sets 
$$\mathcal{S}=\left\{
u\in X:
 u\text{ is a least energy positive radial solution of }\eqref{eq:1.1}
\right\},$$
and
$$\widetilde{\mathcal{S}}=\left\{
u\in X:
 u\text{ is a positive radial solution of }\eqref{eq:1.1}
\right\}.$$
We define $
Y=
\left\{
\int_{\mathbb R^N}u^2\,dx:
 u\in \mathcal{S}
\right\},$ 
and
$ 
\widetilde Y=
\left\{
\int_{\mathbb R^N}u^2\,dx:
 u\in \widetilde{\mathcal{S}}
\right\}.$

The purpose of this paper is to prove compactness of \(\mathcal S\) in the
natural space \(X\) and compactness of \(\widetilde{\mathcal S}\) in
\(L^2(\mathbb R^N)\), and consequently compactness of the corresponding mass
sets \(Y\) and \(\widetilde Y\). The main difficulty is to obtain
uniform tail estimates for $\mathcal{S}$ and $\widetilde{\mathcal{S}}$. We establish these estimates by two
distinct methods. For $\mathcal{S}$, we use the stable manifold of the
zero equilibrium for the associated Emden--Fowler system. For $\widetilde{\mathcal{S}}$, we combine a localized Pohozaev functional with a bootstrap argument on
the exterior region. To the best of our knowledge, this is the first work to establish compactness
of positive radial solution sets, and of their \(L^2\)-mass sets, for
``zero mass'' quasi-linear Schr\"odinger equations in a setting where uniqueness
of positive radial solutions is not assumed. 

Let us first consider the following assumptions for the subcritical case.

\medskip
\noindent\textup{(A1)} \(g\in C^1([0,\infty))\), \(g(0)=0\), and \(g(s)>0\) for \(s>0\).

\medskip
\noindent\textup{(A2)} There exist \(m>2^*\), \(s_0>0\), and \(C_0>0\) such that
\(0\le g(s)\le C_0s^{m-1}\) and \(|g'(s)|\le C_0s^{m-2}\) for
\(s\in(0,s_0]\).

\medskip
\noindent\textup{(A3)} \(\lim\limits_{s\to+\infty} \frac{g(s)}{s^{2\cdot 2^*-1}}=0\)

\medskip
\noindent\textup{(A4)} There exist \(s_1>0\) and \(\nu>2^*\) such that
\(\nu G(s)\le s g(s)\) for any \(s\in(0,s_1]\).

\medskip
\noindent\textup{(A5)} There exist \(A_0>0\), \(\ell>2\), and \(C_1>0\), independent of \(a\), such that, for any
\(a\ge A_0\), the function
\[
\overline{g}_a(t)=\frac{g(at)}{g(a)},\quad 0\le t\le1,
\]
belongs to \(C([0,1])\cap C^1((0,1])\) and satisfies
\[
0\le \overline{g}_a(t)\le C_1t^{\ell-1},\quad
|\overline{g}_a'(t)|\le C_1t^{\ell-2}
\quad\text{for }0<t\le1.
\]

\medskip
\noindent\textup{(A6)} There exist \(p\in(4,2\cdot 2^*)\), \(\sigma_0\in(0,1)\),
\(C_2>0\), and \(S_1>0\) such that
\(g(\sigma t)\ge C_2\sigma^{p-1}g(t)\) for \(t\ge S_1\) and
\(0<\sigma\le\sigma_0\).

\medskip
\noindent\textup{(A7)}
\(\limsup\limits_{t\to+\infty} \frac{t g(t)}{G(t)}
\le 2\cdot 2^*\).

\begin{theorem}\label{thm:1.3}
Let \(N\ge5\). Then we have the following conclusions.
\begin{enumerate}
\item[\textup{$(i)$}]
If \textup{(A1)}--\textup{(A3)} hold, then \(\mathcal S\) is a nonempty compact
subset of \(X\). Consequently, \(Y\) is a nonempty compact subset of
\((0,\infty)\).

\item[\textup{$(ii)$}]
If \textup{(A1)}, \textup{(A2)}, and \textup{(A4)}--\textup{(A7)} hold, then
\(\widetilde{\mathcal S}\) is a nonempty compact subset of
\(L^2(\mathbb R^N)\). Consequently, \(\widetilde Y\) is a nonempty compact
subset of \((0,\infty)\).
\end{enumerate}
\end{theorem}

We use the dual method established in
\cite{MR2029068,MR1933335}. Under \textup{(A1)}--\textup{(A3)}, the associated ``zero mass'' semilinear equation satisfies the assumptions of
\cite[Theorem~4]{MR695535}. Consequently, \(\mathcal S\) is nonempty. In the
Emden--Fowler variables, assumption \textup{(A2)} yields an asymptotically
autonomous system near the zero equilibrium, which yields a local stable
manifold at the zero equilibrium and uniform tail estimates for the
corresponding least energy positive radial solutions.

Assumption \textup{(A4)} gives the sign required in the localized Pohozaev identity for sufficiently large radii. In
part \textup{$(ii)$}, assumption \textup{(A3)} is not imposed separately, because
it follows from \textup{(A1)} and \textup{(A6)}, as shown in
Lemma~\ref{lem:prelim_A_consequences} \textup{$(ii)$}. Assumptions \textup{(A5)}--\textup{(A7)} allow us to apply the nonexistence result \cite[Theorem~2.2 $(ii)$]{MR4701352} to a rescaled limit, ruling out unbounded sequences of shooting parameters defined below.

\begin{remark}
Here we give some examples of nonlinearities satisfying our conditions. 

(1) \(g(s)=s^{2\cdot2^*-1}/(1+\ln(1+s))\) satisfies (A1)--(A3).

(2) If \(2^*<p<2\cdot2^*\), then
\(g(s)=(2+\sin s)s^{p-1}\) satisfies
\textup{(A1)}--\textup{(A3)}.

(3) If \(4<p<2\cdot2^*\), then \(g(s)=s^{p-1}/(1+s)^2\) satisfies (A1), (A2) and (A4)--(A7).
\end{remark}

We also consider an asymptotically critical case, in which \textup{(A3)} is replaced by a critical asymptotic condition together with a positive lower-order perturbation. We introduce the following assumptions.

\medskip
\noindent\textup{(G1)} \(g\in C^1([0,\infty))\), and \(g(s)\ge0\) for \(s\ge0\).

\medskip
\noindent\textup{(G2)} There exist \(m>2^*\), \(C_0>0\), and \(s_0\in(0,1)\) such that
\[
0\le g(s)\le C_0s^{m-1},
\quad
|g'(s)|\le C_0s^{m-2}
\quad\text{for }s\in[0,s_0].
\]

\medskip
\noindent\textup{(G3)} There exist \(S_\infty>0\), $
2\cdot 2^*-2<q_1\le q_2<2\cdot 2^*$, and
\(c_1,c_2>0\) such that, for any \(s\ge S_\infty\),
\[
c_1s^{q_1-1}
\le g(s)-s^{2\cdot 2^*-1}
\le c_2s^{q_2-1}.
\]

\begin{theorem}\label{thm:1.5}
Let \(N\ge5\). Assume that \textup{(G1)}--\textup{(G3)} hold. Then \(\mathcal S\) is a nonempty compact subset of \(X\). Consequently,
\(Y\) is a nonempty compact subset of \((0,\infty)\).
\end{theorem}

\begin{remark}
Here we give some examples of nonlinearities satisfying (G1)--(G3). 
Let \(2\cdot2^*-2<q<2\cdot2^*\).  

(1) $g(s)=s^{2\cdot 2^*-1}+s^{q-1}$ for $s\ge 0$.

(2) $g(s)=s^{2\cdot 2^*-1}+(2+\sin s)s^{q-1}$ for $s\ge 0$.

(3) $
g(s)=s^{2\cdot 2^*-1}+(1+\ln(1+s))s^{q-1}$ for $s\ge 0$.
\end{remark}

Since only positive solutions are considered, under either
\textup{(A1)}--\textup{(A2)} or \textup{(G1)}--\textup{(G2)}, we extend
\(g\) to \(\mathbb R\) by setting $
g(s)=0$ for $s\le0.$
Under 
the change of variables \(u=r(v)\) introduced in \cite{MR2029068,MR1933335}, we consider the dual problem
\begin{equation}\label{eq:intro_dual}
\left\{
\begin{array}{l}
-\Delta v=k(v)\quad \text{in }\mathbb R^N,\\
v\in D^{1,2}(\mathbb R^N),
\end{array}
\right.
\end{equation}
where \(k(t)=r'(t)g(r(t))\). The change of variables \(u=r(v)\) yields a one-to-one correspondence
between weak solutions \(v\in H^1(\mathbb R^N)\) of
\eqref{eq:intro_dual} and weak solutions \(u\in X\) of
\eqref{eq:1.1}. 

For positive radial solutions, \eqref{eq:intro_dual} reduces to the shooting problem
\begin{equation}\label{eq:intro_radial_ode}
\begin{cases}
-v''-\dfrac{N-1}{\rho}v'=k(v),\\
v(0)=\alpha>0,\quad v'(0)=0.
\end{cases}
\end{equation}
For each \(\alpha>0\), let \(v_\alpha\) denote the corresponding solution of
\eqref{eq:intro_radial_ode}. We
consider the sets of shooting parameters
\begin{equation}
Z=
\left\{
\alpha>0:
 v_\alpha\in D^{1,2}(\mathbb R^N)\ \text{and}\ r(v_\alpha)\in\mathcal S
\right\},
\end{equation}
and
\[
\widetilde Z=
\left\{
\alpha>0:
 v_\alpha\in H^1(\mathbb R^N) \text{ is a positive radial solution of }\eqref{eq:intro_dual}
\right\}.
\]
We define the solution maps
\begin{equation}\label{eq:intro_solution_maps}
\begin{aligned}
U&: Z\to X,
& U(\alpha)&=r(v_\alpha),\\
\widetilde U&: \widetilde Z\to L^2(\mathbb R^N),
& \widetilde U(\alpha)&=r(v_\alpha),
\end{aligned}
\end{equation}
and the associated mass maps
\begin{equation}\label{eq:intro_mass_maps}
\begin{aligned}
M&: Z\to(0,\infty),
& M(\alpha)&=\|U(\alpha)\|_{L^2(\mathbb{R}^N)}^2,\\
\widetilde M&: \widetilde Z\to(0,\infty),
& \widetilde M(\alpha)&=\|\widetilde U(\alpha)\|_{L^2(\mathbb{R}^N)}^2.
\end{aligned}
\end{equation} It therefore remains to prove the compactness of $Z$, $\widetilde{Z}$ and the continuity of $U$ and $\tilde{U}$. 

\begin{remark}\label{rem:intro_radial_monotonicity}
Assume the hypotheses of Theorem~\ref{thm:1.3} or of
Theorem~\ref{thm:1.5}. 
If \(v_\alpha\) is a positive radial solution of
\eqref{eq:intro_radial_ode}, then
\[
\rho^{N-1}v_\alpha'(\rho)
=-\int_0^\rho s^{N-1}k(v_\alpha(s))\,ds\le0
\quad\text{for }\rho>0.
\]
Thus \(v_\alpha\) is nonincreasing. If \(v_\alpha\) is constant on a nonempty
interval, uniqueness for \eqref{eq:intro_radial_ode} implies
\(v_\alpha\equiv v_\alpha(0)>0\), which is incompatible with
\(v_\alpha\in D^{1,2}(\R^N)\). Hence
\(v_\alpha\) is strictly decreasing on \((0,\infty)\).
\end{remark}

The paper is organized as follows. Section~\ref{sec:preliminaries} gives some preliminary results. The boundedness of \(Z\) and \(\widetilde Z\) is proved in Section~\ref{sec:gen-boundedness}. Under the assumptions of
Theorem~\ref{thm:1.3} \textup{$(i)$}, the boundedness of \(Z\) is obtained from
direct estimates for \eqref{eq:intro_radial_ode}. The boundedness of
\(\widetilde Z\) under the assumptions of
Theorem~\ref{thm:1.3} \textup{$(ii)$} and that of \(Z\) under the assumptions of
Theorem~\ref{thm:1.5} are proved by blow-up arguments. In Section~\ref{sec:mass-geometry}, we prove the compactness of
\(Z\) and the continuity of \(U\). We apply a local stable manifold theorem
to the associated nonautonomous Emden--Fowler system to obtain the required
uniform decay estimates, thereby completing the proofs of
Theorems~\ref{thm:1.3} \textup{$(i)$} and~\ref{thm:1.5}. In
Section~\ref{sec:finite-mass-radial}, we introduce $P_v$ and combine its
monotonicity with a bootstrap argument to prove that $\widetilde Z$ is
compact and derive uniform decay estimates. The continuity of
$\widetilde U$ then follows from these estimates, completing the proof of
Theorem~\ref{thm:1.3} \textup{$(ii)$}. Finally, in
Appendix~\ref{app:gen-existence}, we establish the mountain pass geometry and
the Palais--Smale compactness below the critical threshold, which yield the
least energy positive solution required in Theorem~\ref{thm:1.5}.

We fix the following notation. 
For \(1\le p\le\infty\), \(\|\cdot\|_p\) denotes the norm in
\(L^p(\R^N)\). The letters \(c\) and \(C\) stand for positive constants whose
precise values may change from line to line. For a real number or a real-valued
function \(w\), set \(w_+=\max\{w,0\}\) and \(w_-=\max\{-w,0\}\). Set
\begin{equation}\label{eq:Estar_def}
\mathcal E_*=\frac1{2N}S^{N/2},
\end{equation} 
where \(S\)
denotes the best constant for the Sobolev embedding
\(D^{1,2}(\mathbb R^N)\hookrightarrow L^{2^*}(\mathbb R^N)\). Let \(E\) denote the least energy level of \(J\) among all positive solutions of
\eqref{eq:intro_dual}.

\section{Preliminaries}\label{sec:preliminaries}

Define \(h\) by
\begin{equation}
h(t)=
\frac12t\sqrt{1+2t^2}
+
\frac{1}{2\sqrt2}
\ln\!\bigl(\sqrt2\,t+\sqrt{1+2t^2}\bigr),
\quad t\ge0, 
\end{equation}
and $h(t)=-h(-t)$ for $t\le 0$. Since
\(h'(t)=\sqrt{1+2t^2}\ge0\) for $t\ge0$, \(h\) has an inverse function \(r: [0,\infty)\rightarrow [0,\infty)\), which is $C^{\infty}$ on $(0, \infty)$ and satisfies the first order Cauchy problem
\[
r'(t)=\frac{1}{\sqrt{1+2r(t)^2}} \quad \text{on }(0,\infty), \quad r(0)=0.
\quad
\]
We now extend $r$ to $\mathbb{R}$ by letting $r(s)=-r(-s)$ when $s<0$. We first present some properties of $r$. 

\begin{lemma}\label{lem:prelim_r_large_bounds}
The odd function \(r\) satisfies the following properties: 
\begin{enumerate}
\item[\textup{$(i)$}]
$r'(0)=1$, $\lim\limits_{t\to0}\frac{r(t)}{t}=1$, and $
\lim\limits_{t\to+\infty}\frac{r(t)}{\sqrt t}=2^{1/4};$
\item[\textup{$(ii)$}] $
0<r'(t)\le1$, $
|r(t)|\le \min\{|t|,2^{1/4}|t|^{1/2}\}$ for $t\in \mathbb{R}$;

\item[\textup{$(iii)$}] $
\frac{1}{2}\le\frac{t r'(t)}{r(t)}\le1$ for $t>0$;
\item[\textup{$(iv)$}] $
\lim\limits_{t\to+\infty}r'(t)t^{1/2}=2^{-3/4};$
\item[\textup{$(v)$}] $
-1\le \frac{t r''(t)}{r'(t)}\le0$ for $t>0$;

\item[\textup{$(vi)$}] $t-\frac{1}{4\sqrt2}\ln t-\frac{1}{\sqrt2}r(t)^2\to b_0>0$ as $t\to +\infty$.
\end{enumerate}
\end{lemma}

\begin{proof}
\textup{$(i)$}--\textup{$(iii)$} follow from
\cite[Lemma~2.1]{MR4433087}, while \textup{$(iv)$} follows from
\cite[Lemma~2.2]{MR3427691}. \textup{$(vi)$} has been proved in
\cite[Lemma~A.1 \textup{(3)}]{MR2983045}. To prove
\textup{$(v)$}, differentiating \(r^{\prime}\) yields $
r''(t)=-\frac{2r(t)(r'(t))^2}{1+2r(t)^2}.$ 
By \textup{$(iii)$}, for \(t>0\), $
-1
\le -\frac{2r(t)^2}{1+2r(t)^2}
\le \frac{t r''(t)}{r'(t)}
\le 0,$ which implies $(v)$. 
\end{proof}

We next prove some properties of \(k\) and \(K\), where $K(t)=\int_{0}^{t}k(s)ds$.

\begin{lemma}\label{lem:prelim_k_local_bounds}
Assume that either \textup{(A1)}--\textup{(A2)} or
\textup{(G1)}--\textup{(G2)} hold. For any \(A>0\), there exists \(C_A>0\)
such that, for \(0<t\le A\),
\[
0\le k(t)\le C_A t^{m-1},
\quad
|k'(t)|\le C_A t^{m-2}.
\]
\end{lemma}

\begin{proof}
By Lemma~\ref{lem:prelim_r_large_bounds} \textup{$(ii)$}, we obtain that, for $t>0$,
\begin{equation}\label{eq:prelim_r_second_derivative_bound}
|r''(t)|
=\frac{2r(t)(r'(t))^2}{1+2r(t)^2}
\le 2t.
\end{equation}
Let \(t_0=\min\{1,s_0\}\). Under either \textup{(A2)} or \textup{(G2)}, we have
\[
0\le g(r(t))\le C_0t^{m-1},
\quad
|g'(r(t))|\le C_0t^{m-2}
\quad\text{for }0<t\le t_0.
\]
Noting that \(k'(t)=r''(t)g(r(t))+(r'(t))^2g'(r(t))\),
Lemma~\ref{lem:prelim_r_large_bounds} \textup{$(ii)$} and
\eqref{eq:prelim_r_second_derivative_bound} yield
\begin{equation}\label{eq:prelim_k_small_bounds}
0\le k(t)\le C_0t^{m-1},
\quad
|k'(t)|
\le 2C_0t^m+C_0t^{m-2}
\le 3C_0t^{m-2}
\quad\text{for }0<t\le t_0.
\end{equation}
By either \textup{(A1)} or \textup{(G1)}, \(k\in C^1((0,\infty))\). This together with \eqref{eq:prelim_k_small_bounds} completes the proof.
\end{proof}

\begin{lemma}\label{lem:prelim_A_consequences}
The following conclusions hold.

\begin{enumerate}
\item[\textup{$(i)$}]
If \textup{(A1)} and \textup{(A3)} hold, then $
\lim\limits_{t\to+\infty}\frac{k(t)}{t^{2^*-1}}=0.$

\item[\textup{$(ii)$}]
If \textup{(A1)} and \textup{(A6)} hold, then
\textup{(A3)} holds.
\end{enumerate}
\end{lemma}

\begin{proof}
By Lemma~\ref{lem:prelim_r_large_bounds} \textup{$(i)$, $(iv)$} and
\textup{(A3)}, we obtain 
\[
\lim\limits_{t\to+\infty}\frac{k(t)}{t^{2^*-1}}
=
\lim\limits_{t\to+\infty}
\frac{g(r(t))}{r(t)^{2\cdot 2^*-1}}
\cdot
\frac{r'(t)r(t)^{2\cdot 2^*-1}}{t^{2^*-1}}
=0.
\]
This proves \textup{$(i)$}.

For \(t\ge S_1/\sigma_0\), \(\sigma=S_1/t\in(0,\sigma_0]\). By (A6), we obtain
\[
g(S_1)=g(\sigma t)\ge C_2\left(\frac{S_1}{t}\right)^{p-1}g(t).
\]
Since \(g(S_1)>0\) by \textup{(A1)}, there exists $C>0$ such that $
g(t)\le C t^{p-1}$ for sufficiently large $t$. This together with \(p<2\cdot 2^*\) yields \textup{(A3)}.
\end{proof}

\begin{lemma}\label{lem:prelim_rescaled_dual}
Assume that \textup{(A1)}, \textup{(A2)}, and \textup{(A5)} hold. Set
\(q=\min\{m,\ell\}>2\). Then there exist \(A_1>1\) and \(C>0\) such that, for \(a\ge A_1\) and \(0<t\le1\),
\begin{equation}\label{eq:prelim_k_ratio}
0\le \frac{k(at)}{k(a)}\le Ct^{q/2-1},
\end{equation}
and
\begin{equation}\label{eq:prelim_k_ratio_derivative}
\left|\frac{d}{dt}\left(\frac{k(at)}{k(a)}\right)\right|
\le Ct^{q/2-2}.
\end{equation}
\end{lemma}

\begin{proof}
It follows from Lemma~\ref{lem:prelim_r_large_bounds} \textup{$(ii)$, $(iv)$} that there exists
\(C>0\) such that $r(s)\le Cs^{1/2}$ and $r'(s)\le Cs^{-1/2}$ 
for $s>0.$ Moreover, by Lemma~\ref{lem:prelim_r_large_bounds} \textup{$(i)$, $(iv)$}, we obtain that there exist \(A_1>1\) and \(c>0\) such that \(r(a)\ge A_0\) for \(a\ge A_1\) and $
r(s)\ge cs^{1/2},$ $r'(s)\ge cs^{-1/2}$ for $s\ge A_1.$ Thus, for \(a\ge A_1\) and \(0<t\le1\),
\begin{equation}\label{eq:prelim_r_rescaled_bounds}
\frac{r(at)}{r(a)}
\le \frac{C(at)^{1/2}}{ca^{1/2}}
\le Ct^{1/2},
\quad
\frac{r'(at)}{r'(a)}
\le \frac{C(at)^{-1/2}}{ca^{-1/2}}
\le Ct^{-1/2}.
\end{equation}

Since \(r\) is increasing, \(0<r(at)/r(a)\le1\). Using \textup{(A5)},
\(q\le\ell\), and \eqref{eq:prelim_r_rescaled_bounds}, we obtain 
\[
0\le \overline{g}_{r(a)}\left(\frac{r(at)}{r(a)}\right)
\le C\left(\frac{r(at)}{r(a)}\right)^{q-1}
\le Ct^{(q-1)/2}
\]
and
\[
\left|\overline{g}_{r(a)}'\left(\frac{r(at)}{r(a)}\right)\right|
\le C\left(\frac{r(at)}{r(a)}\right)^{q-2}
\le Ct^{(q-2)/2}.
\]
Combining this with \eqref{eq:prelim_r_rescaled_bounds}, we obtain
\[
\frac{k(at)}{k(a)}
=
\frac{r'(at)}{r'(a)}
\overline{g}_{r(a)}\left(\frac{r(at)}{r(a)}\right)
\le Ct^{q/2-1},
\]
which proves \eqref{eq:prelim_k_ratio}.

A direct computation yields  
\[
\frac{d}{dt}(\frac{k(at)}{k(a)})
=
\frac{a r''(at)}{r'(a)}
\overline{g}_{r(a)}(\frac{r(at)}{r(a)})
+
\frac{a(r'(at))^2}{r'(a)r(a)}
\overline{g}_{r(a)}'(\frac{r(at)}{r(a)}).
\]
Moreover, by Lemma~\ref{lem:prelim_r_large_bounds} \textup{$(iii)$, $(v)$} and
\eqref{eq:prelim_r_rescaled_bounds}, we obtain that, for any $a\ge A_1$ and $t\in (0,1]$,
\[
\left|\frac{a r''(at)}{r'(a)}\right|
=
\frac1t
\left|\frac{at r''(at)}{r'(at)}\right|
\frac{r'(at)}{r'(a)}
\le Ct^{-3/2},
\]
and
\[
\frac{a(r'(at))^2}{r'(a)r(a)}
=
\frac{a r'(a)}{r(a)}
\left(\frac{r'(at)}{r'(a)}\right)^2
\le Ct^{-1}, 
\]
which, together with \eqref{eq:prelim_r_rescaled_bounds}, yields
\[
\left|
\frac{d}{dt}\left(\frac{k(at)}{k(a)}\right)
\right|
\le
Ct^{-3/2}t^{(q-1)/2}
+
Ct^{-1}t^{(q-2)/2}
\le Ct^{q/2-2}. 
\]
This proves \eqref{eq:prelim_k_ratio_derivative}.
\end{proof}

Finally, We introduce the decomposition used in the asymptotically critical case. 
Fix \(R_\chi\ge S_\infty\), where \(S_\infty\) is as in \textup{(G3)}. Set
\(g_{\mathrm p}(s)=g(s)-s^{2\cdot 2^*-1}\) for \(s\ge0\). Let 
\(\chi\in C^\infty([0,\infty))\) be a cut-off function satisfying $
0\le\chi(s)\le1$, $\chi(s)=0$ if $0\le s\le R_\chi$ and $\chi(s)=1$ for $s\ge2R_\chi.$
For \(t\ge0\), define
\[
k_c(t)=r'(t)r(t)^{2\cdot 2^*-1},
\quad
k_{\mathrm p}(t)=r'(t)\chi(r(t))g_{\mathrm p}(r(t)),
\quad
k_{\mathrm r}(t)=k(t)-k_c(t)-k_{\mathrm p}(t),
\]
and
\[
K_c(t)=\int_0^t k_c(\tau)\,d\tau,
\quad
K_{\mathrm p}(t)=\int_0^t k_{\mathrm p}(\tau)\,d\tau,
\quad
K_{\mathrm r}(t)=\int_0^t k_{\mathrm r}(\tau)\,d\tau.
\]
Then \(k=k_c+k_{\mathrm p}+k_{\mathrm r}\) and
\(K=K_c+K_{\mathrm p}+K_{\mathrm r}\) on \([0,\infty)\).

\begin{lemma}\label{lem:prelim_gen_dual_decomposition}
Assume that \textup{(G1)}--\textup{(G3)} hold. Then the following properties hold.
\begin{itemize}
\item[$(i)$] \(k_{\mathrm r}\in C([0,\infty))\) and has compact
support.
\item[$(ii)$] There exist \(c,C, T_*>0\) such that $
c t^{q_1/2-1}
\le k_{\mathrm p}(t)
\le C t^{q_2/2-1}$ and $
c t^{q_1/2}-C
\le K_{\mathrm p}(t)
\le C(1+t^{q_2/2})$ for \(t>T_*\).

\item[$(iii)$] As \(t\to+\infty\), $$
K_c(t)
=
\frac{2^{\frac{2}{N-2}}}{2^*}t^{2^*}
-\frac{2^{\frac{2^*-1}{2} }}{8}t^{2^*-1}\ln t
+O(t^{2^*-1})$$ 
and $k_c(t)=2^{2/(N-2)}t^{2^*-1}+O(t^{2^*-2}\ln t).$
\end{itemize}
\end{lemma}

\begin{proof}
By Lemma \ref{lem:prelim_r_large_bounds} $(i)$, we can choose \(T_*>0\) sufficiently large that \(r(t)\ge 2 R_{\chi} \ge S_\infty\) for $t\ge T_*$. Thus \(k_{\mathrm p}(t)=r'(t)g_{\mathrm p}(r(t))\) and $k_{\mathrm r}(t)=0$ for \(t\ge T_*\). Combining this with (G1), $(i)$ follows.

It follows from \textup{(G3)}
and Lemma~\ref{lem:prelim_r_large_bounds} \textup{$(i)$, $(iv)$} that there exist $c,C>0$ such that
\begin{equation}\label{eq:prelim_kp_bounds}
c t^{q_1/2-1}
\le k_{\mathrm p}(t)
\le C t^{q_2/2-1}\quad\text{for }t\ge T_*.
\end{equation}
Integrating \eqref{eq:prelim_kp_bounds} over \([T_*,t]\) yields $
c t^{q_1/2}-C
\le K_{\mathrm p}(t)
\le C(1+t^{q_2/2})$ for $t\ge T_*.$ This proves $(ii)$.

By Lemma~\ref{lem:prelim_r_large_bounds} \textup{$(vi)$} and \(r'(t)=(1+2r(t)^2)^{-1/2}\), we obtain 
\[
r(t)^{2\cdot 2^*}
=2^{\frac{2^*}{2}}t^{2^*}
-\frac{2^*}{4}2^{\frac{2^*-1}{2}}t^{2^*-1}\ln t
+O(t^{2^*-1}),
\]
and
\[
k_c(t)=r'(t)r(t)^{2\cdot 2^*-1}
=\frac1{\sqrt2}r(t)^{2\cdot 2^*-2}\bigl(1+O(r(t)^{-2})\bigr)=2^{2/(N-2)}t^{2^*-1}+O(t^{2^*-2}\ln t),
\] which implies $(iii)$.
\end{proof}

\section{Boundedness of \texorpdfstring{$Z$}{Z} and \texorpdfstring{$\widetilde Z$}{Z tilde}}\label{sec:gen-boundedness}

In this section, we prove the boundedness of \(Z\) and \(\widetilde Z\) under
the assumptions of Theorems~\ref{thm:1.3} and \ref{thm:1.5}.
Let \(v_\alpha\) denote the unique radial classical solution of
\eqref{eq:intro_radial_ode}. We first establish the regularity of positive radial weak solutions.

\begin{lemma}
Assume the hypotheses of Theorem~\ref{thm:1.3} or of
Theorem~\ref{thm:1.5}. If \(v\in D^{1,2}(\R^N)\) is a positive radial weak
solution of \eqref{eq:intro_dual}, then $
v\in L^\infty(\R^N)\cap C^2(\R^N).$
\end{lemma}

\begin{proof}
Under the assumptions of Theorem~\ref{thm:1.3} \textup{$(i)$},
Lemma~\ref{lem:prelim_A_consequences} \textup{$(i)$} yields
\(k(s)=o(s^{2^*-1})\) as \(s\to+\infty\). Under the assumptions of
Theorem~\ref{thm:1.3} \textup{$(ii)$},
Lemma~\ref{lem:prelim_A_consequences} \textup{$(ii)$} first yields
\textup{(A3)}, and Lemma~\ref{lem:prelim_A_consequences} \textup{$(i)$} then applies. Under the assumptions of Theorem~\ref{thm:1.5}, Lemma~\ref{lem:prelim_gen_dual_decomposition}
yields \(k(s)\le Cs^{2^*-1}\) for all sufficiently large \(s\). This together with
Lemma~\ref{lem:prelim_k_local_bounds} yields 
\(k(s)\le C s^{2^*-1}\) for \(s>0\). 
Moreover, \eqref{eq:prelim_k_small_bounds} and the zero extension of \(g\) 
yield
\(k\in C^1(\mathbb R)\). Standard elliptic regularity yields
\(v\in L^\infty(\R^N)\cap C^2(\R^N)\).
\end{proof}

\begin{lemma}\label{lem:solution_sets_nonempty_decay}
Let \(N\ge5\).
\begin{enumerate}
\item[\textup{$(i)$}]
If \textup{(A1)}--\textup{(A3)} or \textup{(G1)}--\textup{(G3)} hold, then
\(\mathcal S\neq\varnothing\), and, for any \(\alpha\in Z\),
\[
v_\alpha(\rho)=O(\rho^{2-N}),
\quad |v_\alpha'(\rho)|=O(\rho^{1-N})
\quad\text{as }\rho\to\infty.
\]
\item[\textup{$(ii)$}]
If \textup{(A1)}, \textup{(A2)}, and
\textup{(A4)}--\textup{(A7)} hold, then
\(\widetilde{\mathcal S}\neq\varnothing\).
\end{enumerate}
\end{lemma}

\begin{proof}
Under \textup{(A1)}--\textup{(A3)}, it follows from Lemmas~\ref{lem:prelim_k_local_bounds} and
\ref{lem:prelim_A_consequences} \textup{$(i)$} that \eqref{eq:intro_dual} satisfies the hypotheses of
\cite[Theorem~4]{MR695535}. Thus, \eqref{eq:intro_dual} admits a least energy positive radial solution $v$. By
Lemmas~\ref{lem:prelim_k_local_bounds} and \ref{lem:prelim_A_consequences} \textup{$(i)$}, we obtain that $0\le K(s)\le Cs^{2^*}$ for $s\ge0.$ \cite[Lemma~4 and Theorem~5]{MR1617704} yields that 
\begin{equation}
v(\rho)\le C\rho^{2-N},
\quad
|v'(\rho)|\le C\rho^{1-N}
\end{equation}
for all sufficiently large \(\rho\). Since \(N\ge5\),
\(v\in H^1(\mathbb R^N)\), and hence \(r(v)\in X\). The same estimates hold with \(v\) replaced by \(v_\alpha\). This proves $(i)$. The same conclusion holds under \textup{(G1)}--\textup{(G3)} by 
Proposition~\ref{prop:sec3_groundstate_identification}. 

By Lemma~\ref{lem:prelim_A_consequences} \textup{$(ii)$}, (A1) and (A6) yield
\textup{(A3)}. Noting that 
\(\mathcal S\subset\widetilde{\mathcal S}\), we have \(\widetilde{\mathcal S}\neq\varnothing\).
\end{proof}

\begin{proposition}\label{prop:Z_bounded_general}
Assume that \textup{(A1)}--\textup{(A3)} hold. Then the set \(Z\) is bounded.
\end{proposition}

\begin{proof}
Suppose by contradiction that \(Z\) is unbounded. Then there exists a sequence 
\(\{\alpha_n\}\subset Z\) such that \(\alpha_n\to+\infty\). Set
\(v_n=v_{\alpha_n}\). By the Pohozaev identity, $
J(v_n)=\frac1N\int_{\R^N}|\nabla v_n|^2\,dx=E.$
Hence \(\int_{\R^N}|\nabla v_n|^2\,dx=NE\), and the Sobolev
inequality yields
\begin{equation}\label{eq:Z_bounded_L2star_short}
\int_{\R^N}v_n^{2^*}\,dx\le C
\quad\text{for all }n.
\end{equation}

By Lemma~\ref{lem:prelim_A_consequences} \textup{$(i)$}, for any \(\varepsilon>0\), there exists \(T_\varepsilon>0\) such that
\(0\le k(t)\le \varepsilon t^{2^*-1}\) for all \(t\ge T_\varepsilon\).
For sufficiently large \(n\), we have \(\alpha_n\ge 2T_\varepsilon\). Define
\(R_n=(N/(4\varepsilon))^{1/2}\alpha_n^{-2/(N-2)}\).
We claim that
\begin{equation}\label{eq:Z_bounded_plateau_short}
v_n(\rho)\ge \frac{\alpha_n}{2}
\quad\text{for all }0\le \rho\le R_n.
\end{equation}
Indeed, suppose that there exists a first point
\(\rho_n\in(0,R_n]\) such that \(v_n(\rho_n)=\alpha_n/2\). Then
\(T_\varepsilon\le \alpha_n/2\le v_n(\rho)\le \alpha_n\)
for \(0\le\rho\le\rho_n\). Therefore, for $0\le\rho\le\rho_n$, 
\begin{equation}\label{eq:Z_bounded_derivative_bound}
-v_n'(\rho)
=
\rho^{1-N}\int_0^\rho s^{N-1}k(v_n(s))\,ds
\le
\frac{\varepsilon}{N}\alpha_n^{2^*-1}\rho.
\end{equation}
Integrating \eqref{eq:Z_bounded_derivative_bound} over
\([0,\rho_n]\), we obtain
\[
\frac{\alpha_n}{2}
=
\alpha_n-v_n(\rho_n)
\le
\frac{\varepsilon}{2N}\alpha_n^{2^*-1}\rho_n^2
\le
\frac{\varepsilon}{2N}\alpha_n^{2^*-1}R_n^2
=
\frac18 \alpha_n.
\]
This contradiction proves
\eqref{eq:Z_bounded_plateau_short}.

It follows from \eqref{eq:Z_bounded_plateau_short} and the definition of
\(R_n\) that
\[
\int_{\R^N}v_n^{2^*}\,dx
\ge
\int_{B_{R_n}(0)}v_n^{2^*}\,dx
\ge
\left(\frac{\alpha_n}{2}\right)^{2^*}|B_{R_n}(0)|
=
C_N\,\varepsilon^{-N/2},
\]
where \(C_N>0\) depends only on \(N\). Since \(\varepsilon>0\) was arbitrary,
the lower bound \(C_N\varepsilon^{-N/2}\) contradicts 
\eqref{eq:Z_bounded_L2star_short}. Hence \(Z\) is
bounded.
\end{proof}

\begin{proposition}\label{prop:Ztilde_bounded_general}
Assume that \textup{(A1)}, \textup{(A2)}, and \textup{(A5)}--\textup{(A7)} hold.
Then the set \(\widetilde Z\) is bounded.
\end{proposition}

\begin{proof}
Suppose by contradiction that \(\widetilde Z\) is unbounded. Then there exists
a sequence \(\{\alpha_n\}\subset \widetilde Z\) such that \(\alpha_n\to+\infty\).
Let \(v_n=v_{\alpha_n}\) and define
\[
d_n=\sqrt{\frac{\alpha_n}{k(\alpha_n)}},
\quad
\omega_n(x)=\frac{1}{\alpha_n}v_n(d_n x),
\quad
h_n(s)=\frac{k(\alpha_n s)}{k(\alpha_n)}
\quad(s\in[0,1]).
\]
Then \(\omega_n\) solves
\begin{equation}\label{eq:prop_Ztilde_bounded_blowup_eq_corrected}
-\Delta \omega_n=h_n(\omega_n)
\quad\text{in }\R^N,
\end{equation}
and, since \(v_n\) is positive and radially decreasing,
\begin{equation}\label{eq:prop_Ztilde_bounded_blowup_bounds_corrected}
\omega_n(0)=1,
\quad
0<\omega_n\le1
\quad\text{in }\R^N.
\end{equation}

We complete the proof in the following five steps.

\medskip
\noindent Step 1. Precompactness of \(\{h_n\}\) in \(C([0,1])\).

Since \(\alpha_n\to+\infty\),
we have \(\alpha_n\ge A_1\) for all sufficiently large \(n\), where \(A_1\) is
the constant in Lemma~\ref{lem:prelim_rescaled_dual}. By
Lemma~\ref{lem:prelim_rescaled_dual}, there exists a constant \(C>0\),
independent of \(n\), such that, for all sufficiently large \(n\),
\begin{equation}\label{eq:prop_Ztilde_bounded_hn_uniform_bound}
0\le h_n(s)\le C s^{q/2-1}\le C,
\quad
|h_n'(s)|\le C s^{q/2-2}
\quad\text{for all }s\in(0,1],
\end{equation}
which implies that if
\(0<s<t\le1\), then

\[
|h_n(t)-h_n(s)|
\le
\int_s^t |h_n'(\xi)|\,d\xi
\le
C\int_s^t \xi^{q/2-2}\,d\xi
=
\frac{C}{q/2-1}\bigl(t^{q/2-1}-s^{q/2-1}\bigr);
\]
while, if \(s=0\), then $
|h_n(t)-h_n(0)|=h_n(t)\le C t^{q/2-1}.$ 
Therefore \(\{h_n\}\) is uniformly bounded and equicontinuous on \([0,1]\).
The Arzel\`a--Ascoli theorem then yields, up to a subsequence, a function
\(h_\infty\in C([0,1])\) such that
\begin{equation}\label{eq:prop_Ztilde_bounded_hn_conv_corrected}
h_n\to h_\infty
\quad\text{in }C([0,1]).
\end{equation}

\medskip
\noindent Step 2. Local precompactness of \(\{\omega_n\}\) and \(-\Delta\omega=h_\infty(\omega)\) in \(\R^N\).

By \eqref{eq:prop_Ztilde_bounded_blowup_bounds_corrected}
and \eqref{eq:prop_Ztilde_bounded_hn_uniform_bound},
\(\{h_n(\omega_n)\}\) is bounded in
\(L^{\infty}(\R^N)\). Standard elliptic estimates then imply that, after
passing to a subsequence, we may assume that, for
some \(\beta\in(0,1)\), there exists
\(\omega\in C^{1,\beta}_{\mathrm{loc}}(\R^N)\) such that
\begin{equation}\label{eq:prop_Ztilde_bounded_omega_conv_corrected}
\omega_n\to\omega
\quad\text{in }C^{1,\beta}_{\mathrm{loc}}(\R^N).
\end{equation}
In particular, $
\omega(0)=1,$ $
0\le \omega\le1$ in $\R^N$,
and \(\omega\) is radial and radially nonincreasing.

Moreover, by \eqref{eq:prop_Ztilde_bounded_hn_conv_corrected}
and~\eqref{eq:prop_Ztilde_bounded_omega_conv_corrected}, $
h_n(\omega_n)\to h_\infty(\omega)$ in $C_{\mathrm{loc}}(\R^N).
$
Passing to the limit in \eqref{eq:prop_Ztilde_bounded_blowup_eq_corrected},
we conclude that \(-\Delta\omega=h_\infty(\omega)\) in \(\R^N\) in the
distributional sense. Since \eqref{eq:prop_Ztilde_bounded_hn_uniform_bound} and
\eqref{eq:prop_Ztilde_bounded_hn_conv_corrected} imply that
\(h_\infty\) is Hölder continuous on \([0,1]\), standard Schauder estimates yield \(\omega\in C^{2}_{\mathrm{loc}}(\R^N)\).
Since
\(\omega(0)=1\), \(\omega\) is nontrivial. By the strong maximum
principle,
\(0<\omega\le1\) in \(\R^N\).

\medskip
\noindent Step 3. \(\displaystyle \liminf\limits_{s\to0^+}\frac{h_\infty(s)}{s^{\frac p2-1}}>0\).

Let \(p\), \(\sigma_0\), \(C_2\), and \(S_1\) be the constants in \textup{(A6)}, and
set \(s_*=\min\{1,\sigma_0^2/4\}\).
Fix \(s\in(0,s_*]\). Since \(\alpha_n\to\infty\), 
\(\alpha_n s\to\infty\). By Lemma~\ref{lem:prelim_r_large_bounds} \textup{$(i)$, $(iv)$}, we obtain
\[
\frac{r(\alpha_n s)}{r(\alpha_n)}
=
\frac{r(\alpha_n s)}{(\alpha_n s)^{1/2}}
\cdot
\frac{\alpha_n^{1/2}}{r(\alpha_n)}
\,s^{1/2}
\rightarrow s^{1/2},
\]
and
\[
\frac{r'(\alpha_n s)}{r'(\alpha_n)}
=
\bigl(r'(\alpha_n s)(\alpha_n s)^{1/2}\bigr)
\bigl(r'(\alpha_n)\alpha_n^{1/2}\bigr)^{-1}
\,s^{-1/2}
\rightarrow s^{-1/2}.
\]
Hence, for sufficiently large \(n\),
\begin{equation}\label{eq:prop_Ztilde_bounded_ratio_estimates_corrected}
\frac12 s^{1/2}
\le
\frac{r(\alpha_n s)}{r(\alpha_n)}
\le
2 s^{1/2}
\le
\sigma_0,
\quad
\frac12 s^{-1/2}
\le
\frac{r'(\alpha_n s)}{r'(\alpha_n)}.
\end{equation}
Notice that \(r(\alpha_n)\ge S_1\) for sufficiently large \(n\). Therefore,
\textup{(A6)} yields
\begin{equation}\label{eq:prop_Ztilde_bounded_g_ratio_lower_corrected}
\frac{g(r(\alpha_n s))}{g(r(\alpha_n))}
\ge
C_2\left(\frac{r(\alpha_n s)}{r(\alpha_n)}\right)^{p-1}
\ge
C_2 2^{-(p-1)} s^{\frac{p-1}{2}}.
\end{equation}
By \eqref{eq:prop_Ztilde_bounded_ratio_estimates_corrected}
and \eqref{eq:prop_Ztilde_bounded_g_ratio_lower_corrected}, we obtain
\[
h_n(s)
=
\frac{r'(\alpha_n s)}{r'(\alpha_n)}
\cdot
\frac{g(r(\alpha_n s))}{g(r(\alpha_n))}
\ge
c_p s^{\frac p2-1}
\]
for some constant \(c_p>0\) independent of \(s\) and \(n\). Letting
\(n\to\infty\) and using \eqref{eq:prop_Ztilde_bounded_hn_conv_corrected}, we
obtain
\[
h_\infty(s)\ge c_p s^{\frac p2-1}
\quad\text{for any }s\in(0,s_*].
\]
Consequently,
\begin{equation}\label{eq:prop_Ztilde_bounded_hinfty_lower_corrected}
\liminf\limits_{s\to0^+}\frac{h_\infty(s)}{s^{\frac p2-1}}>0.
\end{equation}

\medskip
\noindent Step 4. \(s\,h_\infty(s)\le 2^*\int_0^s h_\infty(\xi)\,d\xi\) for \(0\le s\le1\).

Since \(K(t)=G(r(t))\), Lemma~\ref{lem:prelim_r_large_bounds} \textup{$(i)$, $(iv)$} and
\textup{(A7)} yield
\[
\limsup\limits_{t\to\infty}\frac{t\,k(t)}{K(t)}
=
\limsup\limits_{t\to\infty}
\left(
\frac{t\,r'(t)}{r(t)}
\cdot
\frac{r(t)\,g(r(t))}{G(r(t))}
\right)
\le 2^*.
\]
Thus, for any \(\varepsilon>0\),
\begin{equation}\label{eq:prop_Ztilde_bounded_tk_K_upper_eps}
t\,k(t)\le(2^*+\varepsilon)K(t)
\quad\text{for all sufficiently large }t.
\end{equation}
Fix \(s\in(0,1]\). Since \(\alpha_n s\to\infty\) and
\(K(\alpha_n s)=\alpha_n k(\alpha_n)\int_0^s h_n(\xi)\,d\xi\),
\eqref{eq:prop_Ztilde_bounded_tk_K_upper_eps} yields,
for all sufficiently large \(n\),
\begin{equation}\label{eq:prop_Ztilde_bounded_hn_pohozaev_eps}
s\,h_n(s)\le (2^*+\varepsilon)\int_0^s h_n(\xi)\,d\xi.
\end{equation}
Passing to the limit in \eqref{eq:prop_Ztilde_bounded_hn_pohozaev_eps}, first as
\(n\to\infty\) by \eqref{eq:prop_Ztilde_bounded_hn_conv_corrected} and then as
\(\varepsilon\to0^+\), we obtain
\begin{equation}\label{eq:prop_Ztilde_bounded_hinfty_pohozaev_corrected}
s\,h_\infty(s)\le 2^*\int_0^s h_\infty(\xi)\,d\xi
\quad\text{for any }s\in[0,1].
\end{equation}

\medskip
\noindent Step 5. Nonexistence for the limiting equation.

Extend \(h_\infty\) by \(\widetilde h_\infty=h_\infty\) on \([0,1]\) and
\(\widetilde h_\infty=h_\infty(1)\) on \([1,\infty)\), and set
\(\widetilde H_\infty(s)=\int_0^s\widetilde h_\infty(\tau)\,d\tau\). Then
\(\widetilde h_\infty\in C([0,\infty))\) and \(\widetilde h_\infty\ge0\). For
\(s\ge1\), \(\widetilde H_\infty(s)=\widetilde H_\infty(1)+(s-1)h_\infty(1)\) and, by \eqref{eq:prop_Ztilde_bounded_hinfty_pohozaev_corrected} with $s=1$, 
\[
s\,\widetilde h_\infty(s)
=s\,h_\infty(1)
\le
2^*\widetilde H_\infty(1)+(s-1)h_\infty(1)
\le
2^*\widetilde H_\infty(s).
\]
Together with \eqref{eq:prop_Ztilde_bounded_hinfty_pohozaev_corrected} on
\([0,1]\), this yields
\begin{equation}\label{eq:prop_Ztilde_bounded_global_pohozaev_corrected}
s\,\widetilde h_\infty(s)\le 2^* \widetilde H_\infty(s)
\quad\text{for any }s\ge0.
\end{equation}
Since \(\widetilde h_\infty=h_\infty\) on \([0,1]\),
\eqref{eq:prop_Ztilde_bounded_hinfty_lower_corrected} also yields
\begin{equation}\label{eq:prop_Ztilde_bounded_tilde_hinfty_lower}
\liminf\limits_{s\to0^+}\frac{\widetilde h_\infty(s)}{s^{\frac p2-1}}>0,
\quad
\frac p2\in(2,2^*).
\end{equation}

Since \(0<\omega\le1\) and \(-\Delta\omega=h_\infty(\omega)\), we have
\[
-\Delta\omega=\widetilde h_\infty(\omega)
\quad\text{in }\R^N.
\]
Thus \(\omega\) is a nontrivial nonnegative radial classical solution of
\(-\Delta u=\widetilde h_\infty(u)\) in \(\R^N\). Since
\eqref{eq:prop_Ztilde_bounded_tilde_hinfty_lower} and
\eqref{eq:prop_Ztilde_bounded_global_pohozaev_corrected} verify the hypotheses
of \cite[Theorem~2.2 (ii)]{MR4701352}, $\omega\equiv 0$ in $\mathbb{R}^N$.
This contradiction proves that \(\widetilde Z\) is bounded.
\end{proof}

We next consider the generalized asymptotically critical case. We first establish a uniform estimate for the
rescaled nonlinearities.

\begin{lemma}\label{lem:gen_hn_bounds}
Assume that \textup{(G1)}--\textup{(G3)} hold. Let \(\alpha_n>0\) and
\(\alpha_n\to+\infty\), and define
\(h_n(s)=k(\alpha_n s)/k(\alpha_n)\) for \(0\le s\le1\).
Then there exist \(n_0\in\N\) and \(C>0\) such that
\begin{equation}\label{eq:gen_hn_global}
0\le h_n(s)\le C s^{2^*-1}
\quad\text{for }n\ge n_0\text{ and }s\in[0,1].
\end{equation}
Moreover, for any \(\delta\in(0,1)\),
\begin{equation}\label{eq:gen_hn_C0_conv}
\lim\limits_{n\to\infty}
\sup_{s\in[\delta,1]}\left|h_n(s)-s^{2^*-1}\right|=0.
\end{equation}
\end{lemma}

\begin{proof}
By Lemma~\ref{lem:prelim_gen_dual_decomposition} \textup{$(i)$--$(iii)$} and
\(q_2/2<2^*\), as \(t\to+\infty\),
\[
k(t)=2^{2/(N-2)}t^{2^*-1}+O(t^{2^*-2}\ln t)+O(t^{q_2/2-1}).
\]
Hence
\begin{equation}\label{eq:gen_k_critical_ratio}
\lim\limits_{t\to+\infty}\frac{k(t)}{2^{2/(N-2)}t^{2^*-1}}=1.
\end{equation}
\textup{(G1)}, Lemma~\ref{lem:prelim_k_local_bounds}, and
\eqref{eq:gen_k_critical_ratio} imply
\begin{equation}\label{eq:gen_k_global_upper}
0\le k(t)\le Ct^{2^*-1}\quad \text{for } t\ge0.
\end{equation}
Moreover, \eqref{eq:gen_k_critical_ratio} yields \(n_0\in\N\) and \(c>0\)
such that
\begin{equation}\label{eq:gen_k_alpha_lower}
k(\alpha_n)\ge c\alpha_n^{2^*-1}
\quad\text{for }n\ge n_0.
\end{equation}
By \eqref{eq:gen_k_global_upper} and \eqref{eq:gen_k_alpha_lower}, for
\(n\ge n_0\) and \(s\in[0,1]\),
\[
0\le h_n(s)
=\frac{k(\alpha_n s)}{k(\alpha_n)}
\le C\frac{(\alpha_n s)^{2^*-1}}{\alpha_n^{2^*-1}}
=Cs^{2^*-1}.
\]
Thus \eqref{eq:gen_hn_global} holds.

Let \(\delta\in(0,1)\). By \eqref{eq:gen_k_critical_ratio},
\[
\lim\limits_{n\to\infty}
\sup_{\tau\ge\delta\alpha_n}
\left|\frac{k(\tau)}{2^{2/(N-2)}\tau^{2^*-1}}-1\right|=0.
\]
In particular,
\begin{equation}\label{eq:gen_k_ratio_alpha_delta}
\lim\limits_{n\to\infty}
\sup_{s\in[\delta,1]}
\left|\frac{k(\alpha_n s)}
{2^{2/(N-2)}(\alpha_n s)^{2^*-1}}-1\right|=0
\end{equation}
and
\begin{equation}\label{eq:gen_k_ratio_alpha}
\lim\limits_{n\to\infty}\frac{k(\alpha_n)}
{2^{2/(N-2)}\alpha_n^{2^*-1}}=1.
\end{equation}
Hence \(k(\alpha_n)/(2^{2/(N-2)}\alpha_n^{2^*-1})\ge1/2\) for 
sufficiently large \(n\), and therefore
\[
\begin{aligned}
\sup_{s\in[\delta,1]}\left|h_n(s)-s^{2^*-1}\right|
&\le
\sup_{s\in[\delta,1]}
s^{2^*-1}
\left|
\frac{k(\alpha_n s)}{2^{2/(N-2)}(\alpha_n s)^{2^*-1}}
\left(\frac{k(\alpha_n)}{2^{2/(N-2)}\alpha_n^{2^*-1}}\right)^{-1}
-1
\right|  \\
&\le
2\sup_{s\in[\delta,1]}
\left|\frac{k(\alpha_n s)}
{2^{2/(N-2)}(\alpha_n s)^{2^*-1}}-1\right|
+2\left|\frac{k(\alpha_n)}
{2^{2/(N-2)}\alpha_n^{2^*-1}}-1\right|.
\end{aligned}
\]
\eqref{eq:gen_k_ratio_alpha_delta} and \eqref{eq:gen_k_ratio_alpha} imply \eqref{eq:gen_hn_C0_conv}.
\end{proof}

\begin{proposition}\label{prop:gen_Z_bounded}
Assume that \textup{(G1)}--\textup{(G3)} hold. Then the shooting parameter set \(Z\) is bounded.
\end{proposition}

\begin{proof}
Proposition~\ref{prop:J_below_cstar} yields
\begin{equation}\label{eq:gen_Z_bounded_level_gap}
E<\mathcal E_*.
\end{equation}
Suppose, by contradiction, that \(Z\) is unbounded and choose \(\alpha_n\in Z\) with
\(\alpha_n\to+\infty\). Set \(v_n=v_{\alpha_n}\) and
\[
d_n=\sqrt{\frac{\alpha_n}{k(\alpha_n)}},
\quad
\omega_n(x)=\alpha_n^{-1}v_n(d_nx).
\]
Since \(\alpha_n\in Z\), the profile \(v_n\) is positive and strictly
decreasing. Hence
\[
\omega_n(0)=1,
\quad
0<\omega_n\le1
\quad\text{in }\R^N.
\]
For \(h_n\) as in Lemma~\ref{lem:gen_hn_bounds}, \(\omega_n\) satisfies
\begin{equation}\label{eq:gen_blowup_eq}
-\Delta\omega_n=h_n(\omega_n)
\quad\text{in }\R^N.
\end{equation}
Since \(0<\omega_n\le1\), \eqref{eq:gen_hn_global} yields, for sufficiently large \(n\),
\begin{equation}\label{eq:gen_hn_omega_global}
0\le h_n(\omega_n)\le C\omega_n^{2^*-1}
\quad\text{in }\R^N.
\end{equation}

By \eqref{eq:gen_hn_omega_global} and standard elliptic estimates, after
passing to a subsequence,
\begin{equation}\label{eq:gen_Z_bounded_C1loc}
\omega_n\to\omega\quad\text{in }C^{1,\beta}_{\mathrm{loc}}(\R^N)
\quad\text{for any }0<\beta<1,
\quad
\omega(0)=1,\quad 0\le\omega\le1.
\end{equation}
Let \(K\) be a compact subset of \(\R^N\). We claim that
\begin{equation}\label{eq:gen_nonlin_limit}
\|h_n(\omega_n)-\omega^{2^*-1}\|_{L^\infty(K)}\to0.
\end{equation}
Indeed, fix \(\delta\in(0,1)\). By \eqref{eq:gen_Z_bounded_C1loc}, for sufficiently
large \(n\), one has \(\omega_n\in[\delta/2,1]\) on
\(K\cap\{\omega\ge\delta\}\) and \(\omega_n\le2\delta\) on
\(K\cap\{\omega<\delta\}\). Hence \eqref{eq:gen_hn_C0_conv} and
\eqref{eq:gen_Z_bounded_C1loc} imply
\(h_n(\omega_n)-\omega^{2^*-1}\to0\) in
\(L^\infty(K\cap\{\omega\ge\delta\})\). On \(K\cap\{\omega<\delta\}\), the inequalities
\(0\le\omega<\delta\), \(\omega_n\le2\delta\), and \eqref{eq:gen_hn_global}
yield
\(\|h_n(\omega_n)-\omega^{2^*-1}\|_{L^\infty(K\cap\{\omega<\delta\})}
\le C\delta^{2^*-1}\) for sufficiently large \(n\). Combining the estimates on the two regions,
\[
\limsup\limits_{n\to\infty}
\|h_n(\omega_n)-\omega^{2^*-1}\|_{L^\infty(K)}
\le C\delta^{2^*-1},
\]
where \(C\) is independent of \(\delta\). Since \(\delta\in(0,1)\) is arbitrary,
\eqref{eq:gen_nonlin_limit} follows.

Letting \(n\to\infty\) in \eqref{eq:gen_blowup_eq} and using
\eqref{eq:gen_nonlin_limit}, we obtain $
-\Delta\omega=\omega^{2^*-1}$ in the sense of distributions on
\(\R^N\). Elliptic regularity and the strong maximum principle yield
\(\omega\in C^2(\R^N)\) and \(\omega>0\). Moreover, since each
\(\omega_n\) is radial, so is \(\omega\). Writing \(\omega=\omega(\rho)\), we
have
\[
-\omega''-\frac{N-1}{\rho}\omega'=\omega^{2^*-1},
\quad
\omega(0)=1,
\quad
\omega'(0)=0.
\]
The uniqueness of the regular radial solution to this initial value problem
therefore yields
\[
\omega(x)=
\left(1+\frac{|x|^2}{N(N-2)}\right)^{-\frac{N-2}{2}}.
\]

By the definition of \(Z\), Proposition~\ref{prop:sec3_groundstate_identification},
and the Pohozaev identity,
\[
J(v_n)=E,
\quad
\int_{\R^N}|\nabla v_n|^2\,dx=NE.
\]
Hence, the definitions of \(d_n\) and \(\omega_n\), together with
Lemma~\ref{lem:prelim_gen_dual_decomposition}
\textup{$(ii)$, $(iii)$}, yield
\[
\int_{\R^N}|\nabla\omega_n|^2\,dx
=NE\Bigl(\frac{k(\alpha_n)}{\alpha_n^{2^*-1}}\Bigr)^{\frac{N-2}{2}}
\rightarrow 2NE.
\]
Thus \(\{\omega_n\}\) is bounded in \(D^{1,2}(\R^N)\). Together with
\eqref{eq:gen_Z_bounded_C1loc}, this yields
\(\omega_n\rightharpoonup\omega\) in \(D^{1,2}(\R^N)\), up to a subsequence. By weak lower
semicontinuity and \cite[Lemma~1.46]{Willem1996},
\(S^{N/2}\le2NE\). Thus
\(E\ge \mathcal E_*\) by \eqref{eq:Estar_def}, contradicting
\eqref{eq:gen_Z_bounded_level_gap}. Therefore \(Z\) is bounded.
\end{proof}

\section{Proofs of Theorems 
\ref{thm:1.3} \textup{$(i)$}
and \ref{thm:1.5}}
\label{sec:mass-geometry}

In this section, we prove Theorems~\ref{thm:1.3} \textup{$(i)$} and
\ref{thm:1.5} in parallel. Hence, for any \(\alpha\in Z\), the definition of \(Z\) and the Pohozaev identity yield
\begin{equation}\label{eq:sec5_ground_state_energy}
\|\nabla v_\alpha\|_2^2=NE.
\end{equation} 

\subsection{Compactness of the shooting parameter set}

\begin{lemma}\label{lem:sec5_alpha_bounds}
Assume that either the hypotheses of Theorem~\ref{thm:1.3} \textup{$(i)$} or
those of Theorem~\ref{thm:1.5} hold. Then there exist constants $
0<a_*<A^*<\infty$
such that \(a_*\le\alpha\le A^*\) for any \(\alpha\in Z\).
\end{lemma}

\begin{proof}
The upper bound follows from Proposition~\ref{prop:Z_bounded_general} under the assumptions of Theorem~\ref{thm:1.3} \textup{$(i)$} and from
Proposition~\ref{prop:gen_Z_bounded} in under the assumptions of Theorem~\ref{thm:1.5}.
Thus there exists \(A^*>0\) such that \(\sup Z\le A^*\).

It remains to prove the positive lower bound. Fix \(\alpha\in Z\) and let
\(v=v_\alpha\). By Lemma~\ref{lem:prelim_k_local_bounds},
\[
0\le k(s)\le C s^{m-1}
\quad\text{for all }s\in[0,A^*].
\]
Since \(0\le v\le\alpha\) and \(m>2^*\),
\[
\|\nabla v\|_2^2
=\int_{\mathbb R^N}k(v)v\,dx
\le C\int_{\mathbb R^N}v^m\,dx
\le C\alpha^{m-2^*}\int_{\mathbb R^N}v^{2^*}\,dx.
\]
By the Sobolev inequality, 
\begin{equation}\label{sobolev_identity}
\|\nabla v\|_2^2
\le
C\,\alpha^{m-2^*}\|v\|_{L^{2^*}}^{2^*}
\le
C\,\alpha^{m-2^*}\|\nabla v\|_2^{2^*}.  
\end{equation}
By \eqref{eq:sec5_ground_state_energy},
\(\|\nabla v\|_2^2=NE>0\). Substituting this into
\eqref{sobolev_identity} yields
\[
NE\le C\,\alpha^{m-2^*}(NE)^{2^*/2}, 
\]
which implies that there exists a constant \(a_*>0\), independent of \(\alpha\in Z\), such that \(\alpha\ge a_*\).
\end{proof}

We next show the continuity of the shooting map with respect to the initial value.

\begin{lemma}\label{lem:sec5_continuity}
Assume that either \textup{(A1)}--\textup{(A2)} or
\textup{(G1)}--\textup{(G2)} holds. Let \(I\subset(0,\infty)\) be a closed interval and \(R>0\). Then the map $
\alpha\mapsto v_\alpha$ 
is continuous from \(I\) into \(C^1([0,R])\).
\end{lemma}

\begin{proof}
This is the standard continuous dependence on initial data for the regularized
first-order system associated with \eqref{eq:intro_radial_ode}.
\end{proof}

\begin{proposition}\label{prop:sec5_Z_closed}
Assume that either the hypotheses of Theorem~\ref{thm:1.3} \textup{$(i)$} or those
of Theorem~\ref{thm:1.5} hold. Then \(Z\) is a compact subset of
\((0,\infty)\).
\end{proposition}

\begin{proof}
By Lemma~\ref{lem:sec5_alpha_bounds}, it suffices to prove that \(Z\) is closed
in \((0,\infty)\). Let \(\{\alpha_n\}\subset Z\) be such that
\(\alpha_n\to\alpha_*\), and set \(v_n=v_{\alpha_n}\). By
\eqref{eq:sec5_ground_state_energy},
\(\|\nabla v_n\|_2^2=NE\).
Thus \(\{v_n\}\) is bounded in \(D^{1,2}_{\mathrm{rad}}(\mathbb R^N)\), and,
after passing to a subsequence, \(v_n\rightharpoonup \widetilde v\) in
\(D^{1,2}(\mathbb R^N)\). On the other hand, by Lemma~\ref{lem:sec5_continuity}, \(v_n\to v_{\alpha_*}\) in \(C^1([0,R])\) for any \(R>0\). Hence
\(\widetilde v=v_{\alpha_*}\) on any compact interval, and therefore $
v_n\rightharpoonup v_{\alpha_*}$ in $D^{1,2}(\mathbb R^N).$ In particular, \(v_{\alpha_*}\in D^{1,2}_{\mathrm{rad}}(\mathbb R^N)\).

Since each \(v_n\) is positive and radially decreasing, \(v_{\alpha_*}\) is
nonnegative and radially decreasing. Moreover, for any
\(\varphi\in C_c^\infty(\mathbb R^N)\), $$
\int_{\mathbb R^N}\nabla v_n\cdot\nabla\varphi\,dx
=\int_{\mathbb R^N}k(v_n)\varphi\,dx.$$ 
Passing to the limit, using weak convergence in \(D^{1,2}\) and local uniform
convergence of \(v_n\), we obtain
\[
\int_{\mathbb R^N}\nabla v_{\alpha_*}\cdot\nabla\varphi\,dx
=
\int_{\mathbb R^N}k(v_{\alpha_*})\varphi\,dx.
\]
Thus \(v_{\alpha_*}\) is a radial weak solution of \eqref{eq:intro_dual}.
By elliptic regularity, \(v_{\alpha_*}\in C^2(\mathbb R^N)\), and the strong
maximum principle yields \(v_{\alpha_*}>0\) in \(\mathbb R^N\).
By weak lower semicontinuity and \eqref{eq:sec5_ground_state_energy}, $$
\|\nabla v_{\alpha_*}\|_2^2
\le
\liminf\limits_{n\rightarrow\infty}\|\nabla v_n\|_2^2
=
NE.$$ The Pohozaev
identity yields
\(J(v_{\alpha_*})=N^{-1}\|\nabla v_{\alpha_*}\|_2^2\le E\). Since \(E\) is the least energy level of \(J\) among all positive solutions of
\eqref{eq:intro_dual}, we also have \(J(v_{\alpha_*})\ge E\). Consequently, $
J(v_{\alpha_*})=E,$ and hence 
\(v_{\alpha_*}\) is a least energy positive solution of
\eqref{eq:intro_dual}. Arguing as in the proof of Lemma \ref{lem:solution_sets_nonempty_decay}, it follows from \cite[Lemma~4 and Theorem~5]{MR1617704} that $v_{\alpha_*}\in H^1(\mathbb{R}^N)$. Note that $J(v_{\alpha_*})=I(r(v_{\alpha_*}))$. Hence \(\alpha_*\in Z\), and \(Z\) is closed. This completes the proof.
\end{proof}

\subsection{Emden--Fowler transformation and uniform decay estimates}

Fix \(\alpha\in Z\). We introduce the Emden--Fowler variables
\[
w_\alpha(t)=e^{\la t}v_\alpha(e^t),
\quad
\la=\frac{N-2}{2},
\quad
\mathbf Y_\alpha(t)=\begin{pmatrix} w_\alpha(t)\\ w_\alpha'(t)\end{pmatrix}.
\]
A direct computation shows that \(w=w_\alpha\) solves
\begin{equation}
w''(t)-\la^2 w(t)+e^{\frac{N+2}{2}t}k\bigl(e^{-\la t}w(t)\bigr)=0, \quad \text{in }\mathbb{R},
\end{equation}
and hence \(\mathbf Y_\alpha\) satisfies the first-order system
\begin{equation}
\mathbf Y'(t)=\Lambda \mathbf Y(t)+\mathbf g(t,\mathbf Y(t)),
\end{equation}
where
\begin{equation}\label{first_order_system_nonlinear}
\Lambda=
\begin{pmatrix}
0&1\\
\la^2&0
\end{pmatrix},
\quad
\mathbf{g}(t,\mathbf Y)=
\begin{pmatrix}
0\\
-\,e^{\frac{N+2}{2}t}k(e^{-\la t}w)
\end{pmatrix}
\quad
\text{and}\quad
\mathbf{Y}=
\begin{pmatrix}
w\\ \eta
\end{pmatrix}.
\end{equation}
To diagonalize the linear part, we choose
\[
P=
\begin{pmatrix}
1&1\\
\la&-\la
\end{pmatrix},
\quad
D=P^{-1}\Lambda P=
\begin{pmatrix}
\la&0\\
0&-\la
\end{pmatrix}.
\]
We then introduce $
\mathbf z(t)=P^{-1}\mathbf Y(t).$
Thus \(\mathbf z_\alpha(t)=P^{-1}\mathbf Y_\alpha(t)\) solves
\begin{equation}\label{eq:sec5_diag_system}
\mathbf z'(t)=D\mathbf z(t)+\widetilde{\mathbf g}(t,\mathbf z(t)),
\end{equation}
where $
\widetilde{\mathbf g}(t,\mathbf z)=P^{-1}\mathbf g(t,P\mathbf z).$

\begin{lemma}\label{lem:sec5_individual_decay}
Assume that either the hypotheses of Theorem~\ref{thm:1.3} \textup{$(i)$} or those
of Theorem~\ref{thm:1.5} hold. Then, for any \(\alpha\in Z\),
\[
\lim\limits_{t\to\infty}\mathbf Y_\alpha(t)=0,
\quad
\lim\limits_{t\to\infty}\mathbf z_\alpha(t)=0.
\]
\end{lemma}

\begin{proof}
By Lemma~\ref{lem:solution_sets_nonempty_decay} \textup{$(i)$} and the definition of $w_{\alpha}$, we obtain
\[
|w_\alpha(t)|+|w_\alpha'(t)|
=O(e^{-\la t})
\quad\text{as }t\to\infty.
\]
Since \(\mathbf z_\alpha(t)=P^{-1}\mathbf Y_\alpha(t)\), it follows that
\[
|\mathbf Y_\alpha(t)|+|\mathbf z_\alpha(t)|=O(e^{-\la t})
\quad\text{as }t\to\infty,
\]
which completes the proof.
\end{proof}

\begin{lemma}\label{lem:sec5_smallness}
Assume that either the hypotheses of Theorem~\ref{thm:1.3} \textup{$(i)$} or those
of Theorem~\ref{thm:1.5} hold. Then there exist \(C_{\mathrm{nl}}>0\) and
\(\delta_0>0\) such that, for any \(|\mathbf z|\le1\) and sufficiently
large \(t\),
\begin{equation}
|\widetilde{\mathbf g}(t,\mathbf z)|
\le
C_{\mathrm{nl}} e^{-\delta_0 t}|\mathbf z|^{m-1},
\quad
|D_{\mathbf z}\widetilde{\mathbf g}(t,\mathbf z)|
\le
C_{\mathrm{nl}} e^{-\delta_0 t}|\mathbf z|^{m-2},
\end{equation}
where $
\delta_0=\la(m-1)-\frac{N+2}{2}>0.$
\end{lemma}

\begin{proof}
Since \(\mathbf Y=P\mathbf z\), it suffices to estimate \(\mathbf g\) in the
original variables. Moreover, 
\(|w|\le \sqrt{2}|\mathbf z|\). By \eqref{first_order_system_nonlinear},
\[
|\mathbf g(t,\mathbf Y)|=e^{\frac{N+2}{2}t}\bigl|k(e^{-\la t}w)\bigr|,\quad
\partial_w\mathbf g(t,\mathbf Y)=
\begin{pmatrix}
0\\
-\,e^{\frac{N+2}{2}t}k'(e^{-\la t}w)e^{-\la t}
\end{pmatrix}.
\]
By the zero extension of \(g\) and Lemma~\ref{lem:prelim_k_local_bounds}, there exists \(C>0\), independent of
\(t\) and \(\mathbf z\), such that, for 
sufficiently large \(t\),
\[
|k(e^{-\la t}w)|\le C e^{-\la(m-1)t}|w|^{m-1},\quad
|k'(e^{-\la t}w)|\le C e^{-\la(m-2)t}|w|^{m-2}.
\]
Hence
\[
|\mathbf g(t,\mathbf Y)|\le C e^{-\delta_0 t}|w|^{m-1}\le C e^{-\delta_0 t}|\mathbf z|^{m-1}
\]
and
\[
|D_{\mathbf Y}\mathbf g(t,\mathbf Y)|\le C e^{-\delta_0 t}|w|^{m-2}\le C e^{-\delta_0 t}|\mathbf z|^{m-2}.
\]
Since \(\widetilde{\mathbf g}(t,\mathbf z)=P^{-1}\mathbf g(t,P\mathbf z)\), we complete the proof.
\end{proof}

Fix $0<\varepsilon<\frac{(N-4)}{2}.$ For \(\rho>0\), let \(B_\rho(0)=\{\mathbf z\in\mathbb R^2:|\mathbf z|<\rho\}\).

\begin{lemma}\label{lem:sec5_local_stable}
Assume that either the hypotheses of Theorem~\ref{thm:1.3} \textup{$(i)$} or
those of Theorem~\ref{thm:1.5} hold. Then there exist \(T_0>0\) and
\(\rho_0\in(0,1)\) such that, for any \(T\ge T_0\), the nonautonomous
system \eqref{eq:sec5_diag_system} has a local stable manifold
\(W_T^s\subset B_{\rho_0}(0)\) at the origin, characterized by  
\begin{equation}\label{eq:sec5_stable_characterization}
W_T^s
=
\left\{
\mathbf z_0=
\begin{pmatrix}
x_0\\
y_0
\end{pmatrix}
\in B_{\rho_0}(0)
\;\middle|\;
\begin{aligned}
&|y_0|<\frac{\rho_0}{2},\quad
\mathbf z(T;T,\mathbf z_0)=\mathbf z_0,\\
&\mathbf z(\cdot;T,\mathbf z_0)
\text{ solves \eqref{eq:sec5_diag_system} on }[T,\infty),\\
&\mathbf z(t;T,\mathbf z_0)\in B_{\rho_0}(0)
\quad\text{for }t\ge T,\\
&\sup_{t\ge T}e^{(\la-\varepsilon)(t-T)}
|\mathbf z(t;T,\mathbf z_0)|<\infty
\end{aligned}
\right\}.
\end{equation}

Moreover, \(W_T^s\) is a
\(C^1\)-graph over the stable eigenspace
\(E^s=\{(0,y):\,y\in\R\}\) of \(D\). 
Furthermore, if \(\mathbf z\) is a solution of \eqref{eq:sec5_diag_system} such
that \(\mathbf z(T)\in W_T^s\), then
\begin{equation}\label{eq:sec5_local_stable_decay}
|\mathbf z(t)|
\le
2|\mathbf z(T)| e^{-(\la-\varepsilon)(t-T)}
\quad\text{for any }t\ge T.
\end{equation}
\end{lemma}

\begin{proof}
This is a standard stable manifold result for nonautonomous perturbations of a
hyperbolic linear system; see, for example,
\cite[Chaps.~4--5]{barreira2008stability}. We give the 
argument for completeness.

Write
\[
\mathbf z(t)=
\begin{pmatrix}
x(t)\\ y(t)
\end{pmatrix},
\quad
\widetilde{\mathbf g}(t,\mathbf z)=
\begin{pmatrix}
f_1(t,\mathbf z)\\
f_2(t,\mathbf z)
\end{pmatrix}.
\]
Then \eqref{eq:sec5_diag_system} is equivalent to
\begin{equation}
\begin{cases}
x'(t)=\la x(t)+f_1(t,\mathbf{z}(t)),\\[1mm]
y'(t)=-\la y(t)+f_2(t,\mathbf{z}(t)).
\end{cases}
\end{equation}

By Lemma~\ref{lem:sec5_smallness}, there exist \(C_{\mathrm{nl}}, T_0>0\) and
\(\delta_0>0\) such that, for \(t\ge T_0\) and 
\(|\mathbf z|\le1\),
\begin{equation}\label{eq:sec5_f_basic}
|\widetilde{\mathbf g}(t,\mathbf z)|
\le
C_{\mathrm{nl}} e^{-\delta_0 t}|\mathbf z|^{m-1},
\quad
|D_{\mathbf z}\widetilde{\mathbf g}(t,\mathbf z)|
\le
C_{\mathrm{nl}} e^{-\delta_0 t}|\mathbf z|^{m-2}.
\end{equation}
After increasing \(T_0\) and decreasing \(\rho_0\in(0,1)\), we may assume that
\begin{equation}\label{eq:sec5_choice_constants_2}
C_{\mathrm{nl}} e^{-\delta_0 T_0}\rho_0^{m-2}
\left(\frac{1}{2\la-\varepsilon+\delta_0}+\frac{1}{\varepsilon}\right)
\le \frac12.
\end{equation}

Fix \(T\ge T_0\). For \(\eta\in\R\) with \(|\eta|<\rho_0/2\), consider
\[
\mathcal X_{*,T}
=
\left\{
\mathbf z\in C([T,\infty),\R^2):
\|\mathbf z\|_{*,T}<\infty
\right\},
\quad
\|\mathbf z\|_{*,T}
=
\sup_{t\ge T}e^{(\la-\varepsilon)(t-T)}|\mathbf z(t)|,
\]
and $
\mathcal B_{*,T}
=
\left\{
\mathbf z\in\mathcal X_{*,T}:
\|\mathbf z\|_{*,T}\le\rho_0
\right\}.$
The space \(\mathcal X_{*,T}\) is a Banach space, and \(\mathcal B_{*,T}\) is closed.
Define \(\Phi_{T,\eta}:\mathcal B_{*,T}\rightarrow\mathcal X_{*,T}\) by
\begin{equation}\label{eq:fixed_point_equation}
(\Phi_{T,\eta}\mathbf z)(t)=
\begin{pmatrix}
-\displaystyle\int_t^\infty e^{\la(t-s)}f_1(s,\mathbf z(s))\,ds\\[4mm]
e^{-\la(t-T)}\eta+\displaystyle\int_T^t e^{-\la(t-s)}f_2(s,\mathbf z(s))\,ds
\end{pmatrix}.  
\end{equation}
Let \(\mathbf z\in\mathcal B_{*,T}\). Since \(|\mathbf z(s)|\le \rho_0\) for $s\ge T$, \eqref{eq:sec5_f_basic} yields
\[
|f_i(s,\mathbf z(s))|
\le
C_{\mathrm{nl}} e^{-\delta_0 s}\rho_0^{m-2}|\mathbf z(s)|,
\quad i=1,2.
\]
By \eqref{eq:fixed_point_equation}, \eqref{eq:sec5_choice_constants_2}, and the
definition of \(\|\cdot\|_{*,T}\),
\begin{equation}\label{eq:sec5_Phi_estimate}
\begin{aligned}
\|\Phi_{T,\eta}\mathbf z\|_{*,T}
&\le
|\eta|+C_{\mathrm{nl}}\rho_0^{m-2}\|\mathbf z\|_{*,T}
\sup_{t\ge T}\int_t^\infty
e^{-(2\la-\varepsilon)(s-t)}e^{-\delta_0s}\,ds\\
&\quad+C_{\mathrm{nl}}\rho_0^{m-2}\|\mathbf z\|_{*,T}
\sup_{t\ge T}\int_T^t
e^{-\varepsilon(t-s)}e^{-\delta_0s}\,ds\\
&\le
|\eta|+C_{\mathrm{nl}}e^{-\delta_0T}\rho_0^{m-2}
\left(
\frac{1}{2\la-\varepsilon+\delta_0}
+\frac{1}{\varepsilon}
\right)
\|\mathbf z\|_{*,T}\\
&\le |\eta|+\frac12\|\mathbf z\|_{*,T}
\le\rho_0.
\end{aligned}
\end{equation}
Thus \(\Phi_{T,\eta}(\mathcal B_{*,T})\subset\mathcal B_{*,T}\).

Next, let \(\mathbf z,\bar{\mathbf z}\in\mathcal B_{*,T}\). By the mean value
theorem and \eqref{eq:sec5_f_basic},
\[
|f_i(s,\mathbf z(s))-f_i(s,\bar{\mathbf z}(s))|
\le
C_{\mathrm{nl}} e^{-\delta_0 s}\rho_0^{m-2}|\mathbf z(s)-\bar{\mathbf z}(s)|,
\quad i=1,2.
\]
Therefore
\[
\|\Phi_{T,\eta}\mathbf z-\Phi_{T,\eta}\bar{\mathbf z}\|_{*,T}
\le\frac12\|\mathbf z-\bar{\mathbf z}\|_{*,T}.
\]
Hence, by Banach fixed point theorem, there exists a unique
\(\mathbf z_{T,\eta}\in\mathcal B_{*,T}\) such that
\(\Phi_{T,\eta}\mathbf z_{T,\eta}=\mathbf z_{T,\eta}\).

Define
\[
h_T(\eta)=-\int_T^\infty e^{\la(T-s)}f_1(s,\mathbf z_{T,\eta}(s))\,ds.
\]
Then
\[
\mathbf z_{T,\eta}(T)=
\begin{pmatrix}
h_T(\eta)\\
\eta
\end{pmatrix}.
\]
Moreover, the map \((\eta,\mathbf z)\mapsto \Phi_{T,\eta}\mathbf z\) is \(C^1\)
on a neighborhood of
\((-\rho_0/2,\rho_0/2)\times\mathcal B_{*,T}\), and the contraction constant
obtained above is uniform with respect to \(\eta\). Hence, by the uniform
contraction principle
\cite[Theorem~1.5]{teschl2001nonlinear}, $
h_T\in C^1((-\rho_0/2,\rho_0/2)).$ 
Since \(f_i(t,0)=0\) for \(i=1,2\), uniqueness of the fixed point yields
\(\mathbf z_{T,0}=0\), and hence \(h_T(0)=0\). Moreover, differentiating
\(\mathbf z_{T,\eta}=\Phi_{T,\eta}\mathbf z_{T,\eta}\) at \(\eta=0\) and using
\(D_{\mathbf z}f_i(t,0)=0\), we obtain \(h_T'(0)=0\).
We therefore set
\[
W_T^s=
\left\{
\begin{pmatrix}
h_T(\eta)\\
\eta
\end{pmatrix}
:\ |\eta|<\frac{\rho_0}{2}
\right\},
\]
which is a \(C^1\)-graph tangent to \(E^s\) at the origin. If
\(\mathbf z(T)\in W_T^s\), then \(\mathbf z=\mathbf z_{T,\eta}\) for some
\(|\eta|<\rho_0/2\). By \eqref{eq:sec5_Phi_estimate}, $
\|\mathbf z\|_{*,T}
\le
|\eta|+\frac12\|\mathbf z\|_{*,T}.$
Hence $
\|\mathbf z\|_{*,T}\le 2|\eta|\le 2|\mathbf z(T)|.$
In particular, \(W_T^s\subset B_{\rho_0}(0)\). Moreover, for \(t\ge T\),
\[
|\mathbf z(t)|
\le
e^{-(\la-\varepsilon)(t-T)}\|\mathbf z\|_{*,T}
\le
2|\mathbf z(T)|e^{-(\la-\varepsilon)(t-T)}.
\]
This proves \eqref{eq:sec5_local_stable_decay}.

It remains to prove the reverse inclusion. Suppose that \(\mathbf z\) satisfies
the conditions on the right-hand side of
\eqref{eq:sec5_stable_characterization}. Writing
\(\mathbf z=(x,y)^{\mathsf T}\), let \(\eta=y(T)\). The variation-of-constants formulas show that \(\mathbf z\) satisfies the integral equation in \eqref{eq:fixed_point_equation}. Repeating the estimate leading to
\eqref{eq:sec5_Phi_estimate}, we obtain $
\|\mathbf z\|_{*,T}
\le |\eta|+\frac12\|\mathbf z\|_{*,T}.$
Thus, \(\|\mathbf z\|_{*,T}\le2|\eta|<\rho_0\). Thus
\(\mathbf z\in\mathcal B_{*,T}\), and uniqueness yields
\(\mathbf z=\mathbf z_{T,\eta}\). Therefore \(\mathbf z(T)\in W_T^s\). This completes the proof.
\end{proof}

\begin{lemma}\label{lem:sec5_decay_implies_local_stable}
Assume that either the hypotheses of Theorem~\ref{thm:1.3} \textup{$(i)$} or
those of Theorem~\ref{thm:1.5} hold. Then there exists \(T_1\ge T_0\) such
that if \(\alpha\in Z\), \(T\ge T_1\), and $
|\mathbf z_\alpha(T)|<\frac{\rho_0}{2},$ then $
\mathbf z_\alpha(T)\in W_T^s.$
\end{lemma}

\begin{proof}
Since \(v_\alpha\in D^{1,2}_{\mathrm{rad}}(\R^N)\) for \(\alpha\in Z\), the
radial lemma yields $
v_\alpha(\rho)\le C_N \|\nabla v_\alpha\|_2\rho^{-\la}.$
Because \(\|\nabla v_\alpha\|_2^2=NE\), there exists a constant
\(M_w>0\), independent of \(\alpha\in Z\), such that
\begin{equation}
0<w_\alpha(t)\le M_w
\quad\text{for any }t\in\R,\ \alpha\in Z.
\end{equation}
Choose \(T_1\ge T_0\) so large that $
M_w e^{-\la T_1}\le 1.$ 
Then, for any \(t\ge T_1\) and every \(\alpha\in Z\),
\[
0\le e^{-\la t}w_\alpha(t)\le 1.
\]
Hence, by Lemma~\ref{lem:prelim_k_local_bounds},
\[
k(e^{-\la t}w_\alpha(t))
\le C e^{-\la(m-1)t}w_\alpha(t)^{m-1}
\le C e^{-\la(m-1)t},
\]
and therefore, recalling the definitions of \(\mathbf g\) and
\(\widetilde{\mathbf g}\), there exists a constant \(C_g>0\), independent of
\(\alpha\in Z\), such that
\begin{equation}\label{eq:sec5_global_forcing_bound}
|\widetilde{\mathbf g}(t,\mathbf z_\alpha(t))|
\le C_g e^{-\delta_0 t}
\quad\text{for any }t\ge T_1,\ \alpha\in Z,
\end{equation}
where \(\delta_0\) is as in Lemma~\ref{lem:sec5_smallness}.

On the other hand, by Lemma~\ref{lem:sec5_smallness}, for \(|\mathbf z|\le \rho_0\),
\begin{equation}\label{eq:sec5_local_forcing_bound}
|\widetilde{\mathbf g}(t,\mathbf z)|
\le C_{\mathrm{nl}} e^{-\delta_0 t}|\mathbf z|^{m-1}
\le C_{\mathrm{nl}}\rho_0^{m-2} e^{-\delta_0 t}|\mathbf z|.
\end{equation}
We claim that, after increasing \(T_1\) if necessary, the following holds: if
\(\alpha\in Z\), \(T\ge T_1\), and $
|\mathbf z_\alpha(T)|<\frac{\rho_0}{2},$
then $
|\mathbf z_\alpha(t)|<\rho_0$ for any $t\ge T.$
Suppose by contradiction that there exists 
\begin{equation}\label{eq:definition_tau}
    \tau=\inf\{t\ge T:\ |\mathbf z_\alpha(t)|=\rho_0\}\in(T,\infty).
\end{equation}
Hence \(|\mathbf z_\alpha(t)|\le \rho_0\) for \(t\in[T,\tau]\). By Lemma~\ref{lem:sec5_individual_decay}, $
|\mathbf z_\alpha(t)|=O(e^{-\la t})$ as $t\to\infty.$
Hence the variation-of-constants formula yields
\eqref{eq:fixed_point_equation} for \(\mathbf z_\alpha\). 
Using \eqref{eq:fixed_point_equation}, together with
\eqref{eq:sec5_local_forcing_bound} on \([T,\tau]\) and
\eqref{eq:sec5_global_forcing_bound} on \([\tau,\infty)\), there exist constants
\(C_1,C_2>0\), independent of \(\alpha\in Z\), such that
\[
\sup_{t\in[T,\tau]}|\mathbf z_\alpha(t)|
\le
|\mathbf z_\alpha(T)|
+
C_1 e^{-\delta_0 T}
+
C_2 \rho_0^{m-2} e^{-\delta_0 T}
\sup_{t\in[T,\tau]}|\mathbf z_\alpha(t)|.
\]
After increasing \(T_1\) once more, we may assume that
\[
C_1 e^{-\delta_0 T_1}\le \frac{\rho_0}{8},
\quad
C_2 \rho_0^{m-2} e^{-\delta_0 T_1}\le \frac14.
\]
Since \(|\mathbf z_\alpha(T)|\le \rho_0/2\), it follows that
\[
\sup_{t\in[T,\tau]}|\mathbf z_\alpha(t)|
\le
\frac{\rho_0}{2}+\frac{\rho_0}{8}
+\frac14 \sup_{t\in[T,\tau]}|\mathbf z_\alpha(t)|,
\]
and hence $
\sup_{t\in[T,\tau]}|\mathbf z_\alpha(t)|\le \frac56\rho_0<\rho_0,$
which contradicts the definition of \(\tau\) in \eqref{eq:definition_tau}. This
proves the claim. The proof of Lemma~\ref{lem:sec5_individual_decay} yields
\(|\mathbf z_\alpha(t)|=O(e^{-\la t})\) as \(t\rightarrow\infty\), and hence
\[
\sup_{t\ge T}e^{(\la-\varepsilon)(t-T)}
|\mathbf z_\alpha(t)|<\infty.
\]
We conclude that \(\mathbf z_\alpha(T)\in W_T^s\).
\end{proof}

\begin{lemma}\label{lem:sec5_uniform_entry}
Assume that either the hypotheses of Theorem~\ref{thm:1.3} \textup{$(i)$} or
those of Theorem~\ref{thm:1.5} hold. Then there exist
\(T_Z\ge T_1\) and \(C_Z>0\) such that
\[
|\mathbf z_\alpha(t)|
\le C_Z e^{-(\la-\varepsilon)(t-T_Z)}
\quad\text{for any }t\ge T_Z,\ \alpha\in Z.
\]
\end{lemma}

\begin{proof}
Set \(\delta_*=\rho_0/8\) and fix \(\alpha\in Z\). By
Lemma~\ref{lem:sec5_individual_decay},
\(\mathbf z_\alpha(t)\rightarrow0\) as \(t\rightarrow\infty\). Hence we may
choose \(T_\alpha\ge T_1\) so large that $
|\mathbf z_\alpha(T_\alpha)|<\delta_*.$ 
By Lemma~\ref{lem:sec5_continuity}, the map
\(\beta\mapsto \mathbf z_\beta(T_\alpha)\) is continuous on \(Z\). Therefore
there exists an open neighborhood \(U_\alpha\subset Z\) of \(\alpha\) such that $
|\mathbf z_\beta(T_\alpha)|<2\delta_*=\frac{\rho_0}{4}$ for any $\beta\in U_\alpha.$ 
Since \(2\delta_*<\rho_0/2\), Lemma~\ref{lem:sec5_decay_implies_local_stable}
implies that $
\mathbf z_\beta(T_\alpha)\in W_{T_\alpha}^s
$ for any $\beta\in U_\alpha.$
Since \(\{U_\alpha\}_{\alpha\in Z}\) is an open cover of the compact set \(Z\),
we can choose a finite subcover $
Z\subset U_{\alpha_1}\cup\cdots\cup U_{\alpha_L},$
and set
\[
T_Z=\max\{T_{\alpha_1},\dots,T_{\alpha_L}\}.
\]
Let \(\tilde{\alpha}\in Z\). We can choose \(j\in\{1,\dots,L\}\) such that
\(\tilde{\alpha}\in U_{\alpha_j}\). Then
\[
\mathbf z_{\tilde{\alpha}}(\tau_j)\in W_{\tau_j}^s,
\quad
|\mathbf z_{\tilde{\alpha}}(\tau_j)|<\frac{\rho_0}{4}.
\]
Applying \eqref{eq:sec5_local_stable_decay}, for any \(t\ge T_Z\), we obtain
\[
|\mathbf z_{\tilde{\alpha}}(t)|
\le
2|\mathbf z_{\tilde{\alpha}}(\tau_j)| e^{-(\la-\varepsilon)(t-\tau_j)}
\le C_Z e^{-(\la-\varepsilon)(t-T_Z)},
\]
where
\[
C_Z=\frac{\rho_0}{2}
\max_{1\le j\le L}e^{(\la-\varepsilon)(T_Z-\tau_j)}.
\]
Since \(\tilde{\alpha}\in Z\) was arbitrary, the proof is complete.
\end{proof}

\begin{proposition}
Assume that either the hypotheses of Theorem~\ref{thm:1.3} \textup{$(i)$} or those
of Theorem~\ref{thm:1.5} hold. Set \(u_\alpha=r(v_\alpha)\). Then there exists
\(C>0\), independent of \(\alpha\in Z\), such that
\begin{equation}\label{eq:sec5_uniform_decay}
u_\alpha(\rho)
\le C\rho^{-(N-2)+\varepsilon}
\quad\text{for any }\rho\ge e^{T_Z},\ \alpha\in Z.
\end{equation}
Consequently,
\begin{equation}\label{eq:sec5_uniform_tail}
\sup_{\alpha\in Z}\int_{|x|>R}u_\alpha(x)^2\,dx\to0
\quad \text{as } R\to\infty.
\end{equation}
\end{proposition}

\begin{proof}
By Lemma~\ref{lem:sec5_uniform_entry} and
\(\mathbf Y_\alpha(t)=P\mathbf z_\alpha(t)\), we have $
|w_\alpha(t)|\le |\mathbf Y_\alpha(t)|
\le C e^{-(\la-\varepsilon)(t-T_Z)}$ for any $t\ge T_Z$ and $\alpha\in Z.$
Let \(\rho\ge e^{T_Z}\). Since $
v_\alpha(\rho)=\rho^{-\la}w_\alpha(\ln\rho),$ 
we obtain
\[
v_\alpha(\rho)
\le
C \rho^{-\la}\rho^{-(\la-\varepsilon)}
=
C\rho^{-(N-2)+\varepsilon}.
\]
It follows from Lemma~\ref{lem:prelim_r_large_bounds} \textup{$(ii)$} that
\[
u_\alpha(\rho)=r(v_\alpha(\rho))\le v_\alpha(\rho)\le C\rho^{-(N-2)+\varepsilon}.
\]
This proves \eqref{eq:sec5_uniform_decay}.

Finally, for any \(R\ge e^{T_Z}\),
\[
\begin{aligned}
\int_{|x|>R}u_\alpha(x)^2\,dx
&=|\mathbb S^{N-1}|\int_R^\infty u_\alpha(\rho)^2\,\rho^{N-1}\,d\rho\\
&\le C\int_R^\infty \rho^{-2(N-2)+2\varepsilon}\rho^{N-1}\,d\rho
= C\int_R^\infty \rho^{-N+3+2\varepsilon}\,d\rho.
\end{aligned}
\]
Since \(\varepsilon<\frac{N-4}{2}\), we have
\(-N+3+2\varepsilon<-1\). Therefore the last integral tends to \(0\) as
\(R\to\infty\), uniformly in \(\alpha\in Z\). This proves
\eqref{eq:sec5_uniform_tail}.
\end{proof}

\subsection[Continuity of the solution map and compactness]{Continuity of the solution map and compactness}

\begin{proposition}\label{prop:sec5_solution_map_continuity}
Assume that either the hypotheses of Theorem~\ref{thm:1.3} \textup{$(i)$} or
those of Theorem~\ref{thm:1.5} hold. Then the solution map \(U\) defined in
\eqref{eq:intro_solution_maps} is continuous. Consequently, the mass map \(M\)
defined in \eqref{eq:intro_mass_maps} is continuous.
\end{proposition}

\begin{proof}
Let \(\alpha_n,\alpha\in Z\) and \(\alpha_n\to\alpha\).
Arguing as in the proof of Proposition~\ref{prop:sec5_Z_closed} and using \eqref{eq:sec5_ground_state_energy}, we obtain that $
v_{\alpha_n}\to v_\alpha$ in $D^{1,2}(\mathbb R^N).$ 
Since
\(\nabla r(v_{\alpha_n})=r'(v_{\alpha_n})\nabla v_{\alpha_n}\) and $r^{\prime}\in (0,1]$, we have
\begin{equation}\label{eq:sec5_solution_map_gradient}
\|\nabla r(v_{\alpha_n})-\nabla r(v_\alpha)\|_2
\le \|\nabla v_{\alpha_n}-\nabla v_\alpha\|_2+
\|\bigl(r'(v_{\alpha_n})-r'(v_\alpha)\bigr)\nabla v_\alpha\|_2
\rightarrow0,
\end{equation}
where the second term converges to zero by
Lemma~\ref{lem:sec5_continuity} and the dominated convergence theorem.

For every \(R>0\), Lemma~\ref{lem:sec5_continuity} yields
\[
\int_{|x|\le R}
\bigl|r(v_{\alpha_n})-r(v_\alpha)\bigr|^2\,dx\to0.
\]
Combining this with \eqref{eq:sec5_uniform_tail}, we obtain
\[
\limsup\limits_{n\to\infty}
\|r(v_{\alpha_n})-r(v_\alpha)\|_2^2
\le
4\sup_{\beta\in Z}\int_{|x|>R}r(v_\beta)^2\,dx
\quad\text{for any }R>0.
\]
Therefore, by \eqref{eq:sec5_uniform_tail},
\begin{equation}\label{eq:sec5_solution_map_L2}
r(v_{\alpha_n})\to r(v_\alpha)\quad\text{in }L^2(\mathbb R^N).
\end{equation}
Combining \eqref{eq:sec5_solution_map_gradient} and
\eqref{eq:sec5_solution_map_L2}, we conclude that $
r(v_{\alpha_n})\to r(v_\alpha)$ in $H^1(\mathbb R^N).$

Since
\[
0\le 2r(t)r'(t)=\frac{2r(t)}{\sqrt{1+2r(t)^2}}\le\sqrt2
\quad\text{for all }t\ge0,
\]
it follows that
\[
\|2r(v_{\alpha_n})r'(v_{\alpha_n})
(\nabla v_{\alpha_n}-\nabla v_\alpha)\|_2\to0.
\]
Furthermore, by Lemma~\ref{lem:sec5_continuity} and the dominated convergence theorem, we obtain 
\[
\|2\bigl(r(v_{\alpha_n})r'(v_{\alpha_n})
-r(v_\alpha)r'(v_\alpha)\bigr)\nabla v_\alpha\|_2\to0.
\]
Consequently, 
\[
\begin{aligned}
& \|\nabla(r(v_{\alpha_n})^2)-\nabla(r(v_\alpha)^2)\|_2 \\ \le    &
\|2r(v_{\alpha_n})r'(v_{\alpha_n})
(\nabla v_{\alpha_n}-\nabla v_\alpha)\|_2+
\|2\bigl(r(v_{\alpha_n})r'(v_{\alpha_n})
-r(v_\alpha)r'(v_\alpha)\bigr)\nabla v_\alpha\|_2
\to 0.
\end{aligned}
\]
Thus $
d_X\bigl(r(v_{\alpha_n}),r(v_\alpha)\bigr)\to0,$
which proves the continuity of \(U\). According to the definition of $M$, \(M\) is continuous.
\end{proof}

\begin{proof}[Proof of Theorem~\ref{thm:1.3} \textup{$(i)$}]
By Lemma~\ref{lem:solution_sets_nonempty_decay}\textup{(i)} and the definition
of \(Z\), we have \(Z\neq\varnothing\). By
Proposition~\ref{prop:sec5_Z_closed}, \(Z\) is compact. By
Proposition~\ref{prop:sec5_solution_map_continuity}, \(U:Z\to X\) is
continuous. Since \(\mathcal S=U(Z)\), the set \(\mathcal S\) is compact in
\(X\). Consequently, \(Y=M(Z)\) is a nonempty compact subset of
\((0,\infty)\).
\end{proof}

\begin{proof}[Proof of Theorem~\ref{thm:1.5}]
By Proposition~\ref{prop:gen_radial_ground_state}, proved in
Appendix~\ref{app:gen-existence}, equation \eqref{eq:1.1} admits a
radially decreasing least energy positive solution. Hence \(Z\neq\varnothing\).
Proposition~\ref{prop:sec5_Z_closed} gives the compactness of \(Z\), and
Proposition~\ref{prop:sec5_solution_map_continuity} gives the continuity of
\(U:Z\to X\). Since \(\mathcal S=U(Z)\), the set \(\mathcal S\) is compact in
\(X\). Consequently, \(Y=M(Z)\) is a nonempty compact subset of
\((0,\infty)\).
\end{proof}

\section{Proof of Theorem~\ref{thm:1.3} \textup{$(ii)$}}\label{sec:finite-mass-radial}

In this section, we prove Theorem~\ref{thm:1.3} \textup{$(ii)$}. Let $v$ be a positive classical radial solution of \eqref{eq:intro_radial_ode}. We introduce the
radial Pohozaev function
\begin{equation}
P_v(\rho)=\rho^N\Bigl(\frac{1}{2}(v^{\prime}(\rho))^2+K(v(\rho))\Bigr)
+\frac{N-2}{2}\rho^{N-1}v(\rho)v'(\rho), \quad \rho>0.
\end{equation}
A direct computation yields
\begin{equation}\label{eq:pohozaev_tail_derivative}
P_v^{\prime}(\rho)=\rho^{N-1}H(v(\rho)),
\quad
H(\xi)=N K(\xi)-\frac{N-2}{2}\,\xi k(\xi).
\end{equation}
According to the Pohozaev identity, if \(v\in H^1(\mathbb R^N)\), then
\begin{equation}\label{eq:pohozaev_identity}
\int_{\mathbb{R}^N}H(v(x))dx=0.
\end{equation}

We first record the tail monotonicity of \(P_v\).

\begin{lemma}\label{lem:pohozaev_tail_monotonicity}
Assume that \textup{(A1)} and \textup{(A4)} hold. Then there exist \(\xi_0>0\) and
\(\theta\in\bigl(0,\frac{N-2}{2N}\bigr)\) such that $
K(\xi)\le \theta\,\xi\,k(\xi)$ for any $\xi\in(0,\xi_0).$
Moreover, if \(v(\rho)\to0\) as \(\rho\to+\infty\), then there exists \(R_0>0\)
such that \(P_v'(\rho)<0\) for any \(\rho\ge R_0\).
\end{lemma}

\begin{proof}
By Lemma~\ref{lem:prelim_r_large_bounds} \textup{$(i)$},
\(\frac{r(\xi)}{\xi r'(\xi)}\to1\) as \(\xi\to0^+\). Hence, there exists $\xi_0>0$ and \(\theta\in(0,\frac{N-2}{2N})\) such that \(r(\xi)\in(0,s_1]\) and $
\frac1\nu\,\frac{r(\xi)}{\xi r'(\xi)}\le\theta$ for $\xi\in(0,\xi_0),$
where \(\nu\) is the constant in \textup{(A4)}. It follows from \textup{(A4)}
that, for \(\xi\in(0,\xi_0)\),
\[
K(\xi)=G(r(\xi))
\le
\frac1\nu\,r(\xi)\,g(r(\xi))
=
\frac1\nu\,\frac{r(\xi)}{\xi\,r'(\xi)}\,\xi\,k(\xi)
\le
\theta\,\xi\,k(\xi).
\]
Therefore,
\[
H(\xi)
=
N K(\xi)-\frac{N-2}{2}\,\xi\,k(\xi)
\le
\left(N\theta-\frac{N-2}{2}\right)\xi\,k(\xi)<0
\quad\text{for any }\xi\in(0,\xi_0).
\]
If \(v(\rho)\to0\) as $\rho\rightarrow +\infty$, then there exists $R_0>0$ such that \(0<v(\rho)<\xi_0\) for any $\rho\ge R_0$. Hence, by \eqref{eq:pohozaev_tail_derivative}, $P_v^{\prime}(\rho)<0$ for $\rho>R_0$. This completes the proof.
\end{proof}

\begin{lemma}\label{lem:tail_bootstrap}
Let \(N\ge5\). Assume that \textup{(A1)}, \textup{(A2)}, and \textup{(A4)} hold. If $
\lim\limits_{\rho\to\infty}P_v(\rho)\ge0,$ and $ 
v(\rho)\to0$ as $\rho\to+\infty,$ then \(v\in H^1(\R^N)\).
\end{lemma}
\begin{proof}
Define 
\begin{equation}\label{eq:y-and-logder}
y(\rho)=-\rho^{N-1}v'(\rho)>0 \quad \text{and} \quad
\mathcal L_v(\rho)=\frac{-\rho v'(\rho)}{v(\rho)}=
\frac{y(\rho)}{\rho^{N-2}v(\rho)}.   
\end{equation}

Since \((\rho^{N-1}v'(\rho))'=-\rho^{N-1}k(v(\rho))\) and \(v^{\prime}(\rho)<0\), we have
\begin{equation}\label{y_monotonicity}
y^{\prime}(\rho)>0 \quad \text{and} \quad y(\rho)>0.
\end{equation}
By Lemma~\ref{lem:pohozaev_tail_monotonicity} and the assumption
\(\lim\limits_{\rho\rightarrow\infty}P_v(\rho)\ge0\), we deduce that
\(P_v(\rho)>0\) for any \(\rho\ge R_0\).
By the definitions of \(y\) and \(\mathcal L_v\), 
\begin{equation}\label{eq:tail_bootstrap_P_rewrite_new}
P_v(\rho)
=
v(\rho)y(\rho)
\Bigl(\frac12 \mathcal L_v(\rho)-\frac{N-2}{2}\Bigr)
+
\rho^N K(v(\rho)).
\end{equation}
Since \(P_v(\rho)>0\) and $$
K(v(\rho))\le \theta\,v(\rho)\,k(v(\rho))
=
\theta\,v(\rho)\,\frac{y'(\rho)}{\rho^{N-1}},$$ 
it follows from \eqref{eq:tail_bootstrap_P_rewrite_new} that
\[
0<
v(\rho)y(\rho)
\Bigl(\frac12 \mathcal L_v(\rho)-\frac{N-2}{2}\Bigr)
+
\theta\,\rho\,v(\rho)\,y'(\rho).
\]
Dividing by \(v(\rho)y(\rho)>0\), we get
\begin{equation}\label{eq:tail_bootstrap_basic_ineq_new}
0<
\frac12 \mathcal L_v(\rho)-\frac{N-2}{2}
+
\theta\,\frac{\rho y'(\rho)}{y(\rho)}
\quad\text{for any }\rho\ge R_0.
\end{equation}
Noting that \(y(\rho)=\rho^{N-2}v(\rho)\mathcal L_v(\rho)\), we have
\[
\frac{\rho y'(\rho)}{y(\rho)}
=
N-2+\frac{\rho v'(\rho)}{v(\rho)}
+\frac{\rho \mathcal L_v'(\rho)}{\mathcal L_v(\rho)}
=
N-2-\mathcal L_v(\rho)
+\frac{\rho \mathcal L_v'(\rho)}{\mathcal L_v(\rho)}.
\]
Substituting this into \eqref{eq:tail_bootstrap_basic_ineq_new}, we infer that
\[
0<
\Bigl(\frac12-\theta\Bigr)\bigl(\mathcal L_v(\rho)-(N-2)\bigr)
+
\theta\,\frac{\rho \mathcal L_v'(\rho)}{\mathcal L_v(\rho)}.
\]
Equivalently,
\begin{equation}\label{eq:tail_bootstrap_logder_ineq}
\mathcal L_v'(\rho)>
\frac{\mu_\theta}{\rho}\,\mathcal L_v(\rho)
\bigl((N-2)-\mathcal L_v(\rho)\bigr)
\quad\text{for any }\rho\ge R_0,
\end{equation}
where \(\mu_\theta=\frac{\frac12-\theta}{\theta}>0\).

We claim that
\begin{equation}\label{eq:tail_bootstrap_liminf_new}
\liminf\limits_{\rho\rightarrow\infty} \mathcal L_v(\rho)\ge N-2.
\end{equation}
Assume by contradiction that there exist $\varepsilon$ and $\rho_n \rightarrow +\infty$ such that $\mathcal L_v(\rho_n)\le N-2-\varepsilon$. By \eqref{eq:tail_bootstrap_logder_ineq}, there exists $R_1\ge R_0$ such that $
\mathcal L_v(\rho)\le N-2-\varepsilon$ for any $\rho\ge R_1.$ Then \eqref{eq:tail_bootstrap_logder_ineq} gives
$
\mathcal L_v'(\rho)\ge \frac{\mu_\theta \varepsilon}{\rho}\mathcal L_v(\rho)$ for any $\rho\ge R_1.$
Since \(\mathcal L_v(\rho)>0\), integrating yields $
\mathcal L_v(\rho)\ge \mathcal L_v(R_1)\Bigl(\frac{\rho}{R_1}\Bigr)^{\mu_\theta\varepsilon}$ for any $\rho\ge R_1.$
This contradicts \(\mathcal L_v(\rho)\le N-2-\varepsilon\). Thus we prove \eqref{eq:tail_bootstrap_liminf_new}.

By \eqref{eq:tail_bootstrap_liminf_new} for every $\varepsilon\in(0,N-2)$ there exists $R_\varepsilon>0$ such that, for $\rho\ge R_\varepsilon$,
$$
\frac{-v'(\rho)}{v(\rho)}=\frac{\mathcal L_v(\rho)}{\rho}\ge \frac{N-2-\varepsilon}{\rho}.
$$ 
Integrating over $[R_\varepsilon,\rho]$, we obtain
\begin{equation}\label{eq:tail_bootstrap_predecay_new}
v(\rho)\le C \rho^{-(N-2-\varepsilon)}
\quad\text{for }\rho\ge R_\varepsilon.
\end{equation}
By Lemma~\ref{lem:prelim_k_local_bounds}, there exists \(R_\varepsilon'>R_\varepsilon\) such that $
k(v(\rho))\le C\,v(\rho)^{m-1}$ for $\rho\ge R_\varepsilon'.$
Using \eqref{eq:tail_bootstrap_predecay_new}, we infer that
\[
y'(\rho)=\rho^{N-1}k(v(\rho))
\le
C\,\rho^{N-1-(m-1)(N-2-\varepsilon)}
\quad\text{for }\rho\ge R_\varepsilon'.
\] Choose \(\varepsilon>0\) so small that
$
(m-1)(N-2-\varepsilon)>N.$ Then,
\[
\int_{R_\varepsilon'}^\infty
\rho^{N-1-(m-1)(N-2-\varepsilon)}\,d\rho<\infty,
\]
which means \(y'\in L^1(R_\varepsilon',\infty)\). By \eqref{y_monotonicity}, it follows that $y(\rho)$ is bounded above. Therefore
\begin{equation}\label{eq:v'_decay}
-v'(\rho)=\frac{y(\rho)}{\rho^{N-1}}\le C\rho^{1-N}.   
\end{equation}
Integrating from \(\rho\) to \(+\infty\), and using \(v(\rho)\rightarrow0\) as $\rho \rightarrow \infty$, we deduce that 
\begin{equation}\label{eq:v_decay}
 v(\rho)\le C\rho^{2-N}.   
\end{equation}
Finally,
\[
\int_{R_0}^{\infty}|v'(\rho)|^2\rho^{N-1}\,d\rho
\le C\int_{R_0}^{\infty}\rho^{-N+1}\,d\rho<\infty,
\]
and, since $N\ge 5$,
\[
\int_{R_0}^{\infty}v(\rho)^2\rho^{N-1}\,d\rho
\le C\int_{R_0}^{\infty}\rho^{-N+3}\,d\rho<\infty.
\]
Therefore $v\in H^1(\mathbb{R}^N)$.
\end{proof}

We next establish a positive lower bound for \(\widetilde Z\).

\begin{lemma}\label{lem:Z_tilde_positive_lower}
Assume that \textup{(A1)}--\textup{(A2)} and
\textup{(A4)}--\textup{(A7)} hold. Then \(\inf\widetilde Z\ge \xi_0\), where
\(\xi_0\) is the constant given by Lemma~\ref{lem:pohozaev_tail_monotonicity}.
\end{lemma}

\begin{proof}
Suppose by contradiction that there exists \(\alpha\in\widetilde Z\) such
that \(0<\alpha<\xi_0\). Since \(v_\alpha\) is decreasing,
\(0<v_\alpha(\rho)<\xi_0\) for any \(\rho>0\). By
Lemma~\ref{lem:pohozaev_tail_monotonicity}, we have $H(v_\alpha(\rho))<0$ for $\rho>0$.
This contradicts \(\int_{\mathbb R^N}H(v_\alpha)\,dx=0\). Therefore,
\(\inf\widetilde Z\ge \xi_0>0\).
\end{proof}

\begin{proposition}
\label{prop:Z_tilde_closed}
Assume that \textup{(A1)}--\textup{(A2)} and
\textup{(A4)}--\textup{(A7)} hold. Then \(\widetilde Z\) is compact in
\((0,\infty)\).
\end{proposition}

\begin{proof}
Let \(\{\alpha_n\}\subset\widetilde Z\) satisfy \(\alpha_n\to\alpha_*>0\), and
set \(v_n=v_{\alpha_n}\). Let \(v_*=v_{\alpha_*}\) be the unique classical
radial solution of \eqref{eq:intro_radial_ode} with initial value
\(v(0)=\alpha_*\) and \(v'(0)=0\). By Lemma~\ref{lem:sec5_continuity}, $
v_n\to v_*$ in $C^1([0,R])$ for any $R>0.$
Since each \(v_n\) is positive and radially decreasing, \(v_*\) is
nonnegative and radially decreasing on \([0,\infty)\). By
the strong maximum principle, \(v_*(\rho)>0\) for \(\rho\ge0\).

Set \(\ell_*=\lim\limits_{\rho\to\infty}v_*(\rho)\ge0\). We first prove that
\(\ell_*=0\). Assume by contradiction that \(\ell_*>0\). Then there exists $\tilde{\rho}$ such that \(v_*(\rho)\ge \ell_*/2\) for \(\rho\ge \tilde{\rho}\). Since \(k\in C^1([0,\infty))\) and $k(s)>0$ on \((0,\infty)\), $
c_0=\min_{s\in[\ell_*/2,\alpha_*]}k(s)>0$ is well defined. Therefore, \(k(v_*(\rho))\ge c_0\) for \(\rho\ge \tilde{\rho}\). Since
\[
(\rho^{N-1}v_*'(\rho))'=-\rho^{N-1}k(v_*(\rho))\le -c_0\rho^{N-1}
\]
for $\rho\ge \tilde{\rho}$, 
integrating over \([\tilde{\rho},\rho]\) yields
\[
\rho^{N-1}v_*'(\rho)-\tilde{\rho}^{N-1}v_*'(\tilde{\rho})
\le
-\frac{c_0}{N}\bigl(\rho^N-\tilde{\rho}^N\bigr).
\]
Since \(v_*\) is decreasing, \(v_*'(\tilde{\rho})\le0\), and thus
\[
v_*'(\rho)\le -\frac{c_0}{N}\left(\rho-\frac{\tilde{\rho}^N}{\rho^{N-1}}\right).
\]
Hence there exists \(\rho_1\ge\tilde{\rho}\) such that $
v_*'(\rho)\le -\frac{c_0}{2N}\rho$ for any $\rho\ge\rho_1.$
Integrating once more over \([\rho_1,\rho]\), we obtain
\[
v_*(\rho)\le v_*(\rho_1)-\frac{c_0}{4N}\bigl(\rho^2-\rho_1^2\bigr),
\]
which tends to \(-\infty\) as \(\rho\to\infty\), a contradiction. Therefore
\(\lim\limits_{\rho\to\infty}v_*(\rho)=0\).

Assume by contradiction that \(\alpha_*\notin\widetilde Z\). Since
\(v_*=v_{\alpha_*}\) is a positive classical solution with
\(v_*(\rho)\to0\), this means that \(v_*\notin H^1(\R^N)\). By
Lemma~\ref{lem:pohozaev_tail_monotonicity}, there exists \(R_1>0\) such that $
P_{v_*}'(\rho)<0$ for any $\rho\ge R_1.$
Hence the limit $
L_P=\lim\limits_{\rho\rightarrow\infty}P_{v_*}(\rho)$
exists in \([-\infty,\infty)\).
If \(L_P\ge0\), then Lemma~\ref{lem:tail_bootstrap} yields
\(v_*\in H^1(\R^N)\), which is impossible. Therefore \(L_P<0\).
Let \(\xi_0>0\) be the constant given by Lemma~\ref{lem:pohozaev_tail_monotonicity}.
Choose \(\rho_*>R_1\) and \(\eta>0\) such that
\begin{equation}
P_{v_*}(\rho_*)\le -2\eta<0,
\quad
v_*(\rho_*)<\frac{\xi_0}{2}.
\end{equation}
Since \(v_n\rightarrow v_*\) in \(C^1([0,\rho_*])\), we have
\[
P_{v_n}(\rho_*)\rightarrow P_{v_*}(\rho_*),
\quad  
v_n(\rho_*)\rightarrow v_*(\rho_*).
\]
Hence, for sufficiently large \(n\),
\begin{equation}\label{pohozaev_tail_negative}
P_{v_n}(\rho_*)\le -\eta<0,
\quad
v_n(\rho_*)<\xi_0.  
\end{equation}
Since \(v_n\) is decreasing, \(v_n(\rho)\le v_n(\rho_*)<\xi_0\) for any
\(\rho\ge\rho_*\). Therefore Lemma~\ref{lem:pohozaev_tail_monotonicity} yields
\(P_{v_n}'(\rho)<0\) for \(\rho\ge\rho_*\). On the other hand, since
\(v_n\in H^1(\R^N)\), integrating \eqref{eq:pohozaev_tail_derivative} from
\(0\) to \(\rho\) and using \eqref{eq:pohozaev_identity} yields $
\lim\limits_{\rho\to\infty}P_{v_n}(\rho)=0.$ 
But \eqref{pohozaev_tail_negative} and the strict decrease of \(P_{v_n}\) on \([\rho_*,\infty)\) imply that
\[
P_{v_n}(\rho)\le -\eta
\quad\text{for any }\rho\ge\rho_*,
\]
which is impossible. Thus \(v_*\in H^1(\R^N)\) and
\(\alpha_*\in\widetilde Z\). Hence \(\widetilde Z\) is closed in
\((0,\infty)\). Combining Proposition~\ref{prop:Ztilde_bounded_general} and
Lemma~\ref{lem:Z_tilde_positive_lower}, we conclude that \(\widetilde Z\) is
compact in \((0,\infty)\).
\end{proof}

\begin{proposition}\label{prop:compact_family_tail_mass}
Assume that \textup{(A1)}--\textup{(A2)} and
\textup{(A4)}--\textup{(A7)} hold, and set \(u_\alpha=r(v_\alpha)\). Then
there exist constants \(R_*>0\) and \(C_*>0\), independent of
\(\alpha\in\widetilde Z\), such that
\begin{equation}\label{eq:compact_family_uniform_tail_v}
v_\alpha(\rho)\le C_* \rho^{2-N},
\quad
u_\alpha(\rho)\le C_* \rho^{2-N}
\quad\text{for }\rho\ge R_*,\ \alpha\in\widetilde Z.
\end{equation}
Consequently,
\begin{equation}\label{eq:compact_family_uniform_tail_u}
\sup_{\alpha\in\widetilde Z}\int_{|x|>R}u_\alpha(x)^2\,dx\rightarrow0
\quad\text{as }R\rightarrow\infty.
\end{equation}
Moreover, the solution map \(\widetilde U\) defined in
\eqref{eq:intro_solution_maps} is continuous. Consequently, the mass map
\(\widetilde M\) defined in \eqref{eq:intro_mass_maps} is continuous.
\end{proposition}

\begin{proof}
By Proposition~\ref{prop:Z_tilde_closed}, the set \(\widetilde Z\) is compact.
Set $
a_*=\min\widetilde Z>0$ and $
A^*=\max\widetilde Z<\infty.$ Let \(\xi_0>0\) and \(\theta\in(0,(N-2)/(2N))\) be given by
Lemma~\ref{lem:pohozaev_tail_monotonicity}, and define $
b=\min\left\{\frac{a_*}{2},\,\frac{\xi_0}{2}\right\}>0.$ Since \(k\) is continuous and positive on \((0,\infty)\), the constants $
m_b=\min_{s\in[b,A^*]}k(s)>0$ and $
M_A=\max_{s\in[0,A^*]}k(s)<\infty$ 
are well defined.

Fix \(\alpha\in\widetilde{Z}\). Since \(v_\alpha(0)=\alpha\ge a_*\ge2b\),
\(v_\alpha(\rho)\rightarrow0\) as \(\rho\rightarrow\infty\), and \(v_\alpha\) is continuous and
strictly decreasing, there exists a unique radius \(\rho_\alpha>0\) such that $
v_\alpha(\rho_\alpha)=b.$
Moreover, $
b\le v_\alpha(\rho)\le A^*$ for any $0\le \rho\le \rho_\alpha.$
We first prove that \(\rho_\alpha\) are uniformly bounded above and below. Since $$-(\rho^{N-1}v_\alpha'(\rho))'=\rho^{N-1}k(v_\alpha(\rho)),$$
we have $
-v_\alpha'(\rho)
=
\rho^{1-N}\int_0^\rho s^{N-1}k(v_\alpha(s))\,ds.$
Using \(k(v_\alpha)\le M_A\) on \([0,\rho_\alpha]\), we obtain $$
-v_\alpha'(\rho)\le \frac{M_A}{N}\rho$$ for $0\le \rho\le \rho_\alpha.$
Integrating from \(0\) to \(\rho_\alpha\), we get
\[
b=v_\alpha(\rho_\alpha)\ge \alpha-\frac{M_A}{2N}\rho_\alpha^2\ge
a_*-\frac{M_A}{2N}\rho_\alpha^2.
\]
Hence, $
\rho_\alpha^2\ge \frac{2N(a_*-b)}{M_A}.$ Similarly, using \(k(v_\alpha)\ge m_b\) on \([0,\rho_\alpha]\), we obtain $
\rho_\alpha^2\le \frac{2N(A^*-b)}{m_b}.$
Thus there exist constants \(0<\rho_-<\rho_+<\infty\), independent of \(\alpha\),
such that
\begin{equation}\label{eq:compact_family_hit_radius_bounds}
\rho_-\le \rho_\alpha\le \rho_+
\quad\text{for }\alpha\in\widetilde{Z}.
\end{equation}

As in \eqref{eq:y-and-logder}, set \[
y_\alpha(\rho)=-\rho^{N-1}v_\alpha'(\rho)>0,
\quad
\mathcal L_\alpha(\rho)=\frac{-\rho v_\alpha'(\rho)}{v_\alpha(\rho)}
=
\frac{y_\alpha(\rho)}{\rho^{N-2}v_\alpha(\rho)}.
\] 
A direct computation gives
\[
y_\alpha(\rho_\alpha)
=
\int_0^{\rho_\alpha}s^{N-1}k(v_\alpha(s))\,ds
\ge
\frac{m_b}{N}\rho_\alpha^N.
\]
Hence
\begin{equation}\label{eq:compact_family_f_lower}
\mathcal L_\alpha(\rho_\alpha)
=
\frac{y_\alpha(\rho_\alpha)}{\rho_\alpha^{N-2}b}
\ge
\frac{m_b}{Nb}\rho_\alpha^2
\ge
\frac{m_b}{Nb}\rho_-^2
=L_*>0.
\end{equation}

Now, for \(\rho\ge \rho_\alpha\), we have \(0<v_\alpha(\rho)\le b\le \xi_0/2<\xi_0\).
Therefore Lemma~\ref{lem:pohozaev_tail_monotonicity} gives $
P_{v_\alpha}'(\rho)<0$ for any $\rho\ge \rho_\alpha.$
Since \(v_\alpha\in H^1(\R^N)\), integrating \eqref{eq:pohozaev_tail_derivative} from $0$ to $\rho$ and using \eqref{eq:pohozaev_identity} yields $
\lim\limits_{\rho\rightarrow\infty}P_{v_\alpha}(\rho)=0.$
Hence $P_{v_\alpha}(\rho)>0$ for any $\rho\ge \rho_\alpha.$ As in the proof of \eqref{eq:tail_bootstrap_logder_ineq}, we have
\begin{equation}\label{eq:compact_family_f_ineq}
\mathcal L_\alpha'(\rho)>
\frac{\mu_{\theta}}{\rho}\,\mathcal L_\alpha(\rho)\bigl((N-2)-\mathcal L_\alpha(\rho)\bigr)
\quad\text{for any }\rho\ge \rho_\alpha.
\end{equation}
Fix \(\varepsilon\in(0,N-2)\). We claim that there exists \(R_\varepsilon>0\),
independent of \(\alpha\in\widetilde{Z}\), such that
\begin{equation}\label{eq:compact_family_f_eventually}
\mathcal L_\alpha(\rho)\ge N-2-\varepsilon
\quad\text{for any }\rho\ge R_\varepsilon,\ \alpha\in\widetilde{Z}.
\end{equation}
Indeed, suppose that for some \(\alpha\in\widetilde{Z}\) and some
\(\rho\ge \rho_\alpha\), one has $
\mathcal L_\alpha(\sigma)\le N-2-\varepsilon$ for any $\sigma\in[\rho_\alpha,\rho].$
Then \eqref{eq:compact_family_f_ineq} yields $
\mathcal L_\alpha'(\sigma)\ge \frac{\mu_{\theta}\varepsilon}{\sigma}\mathcal L_\alpha(\sigma)$ for any $\sigma\in[\rho_\alpha,\rho].$ It follows from \(\mathcal L_\alpha>0\), \eqref{eq:compact_family_f_lower} and
\eqref{eq:compact_family_hit_radius_bounds} that 
\begin{equation}\label{eq:f_increasing}
\mathcal L_\alpha(\rho)\ge \mathcal L_\alpha(\rho_\alpha)
\Bigl(\frac{\rho}{\rho_\alpha}\Bigr)^{\mu_{\theta}\varepsilon}\ge
L_*
\Bigl(\frac{\rho}{\rho_+}\Bigr)^{\mu_{\theta}\varepsilon}. 
\end{equation}
Choose \(R_\varepsilon\ge \rho_+\) so large that $
L_*
\Bigl(\frac{R_\varepsilon}{\rho_+}\Bigr)^{\mu_{\theta}\varepsilon}\ge N-2-\varepsilon.$
If for some \(\alpha\in\widetilde{Z}\), one had $
\mathcal L_\alpha(\sigma)<N-2-\varepsilon$ for $\sigma\in[\rho_\alpha,R_\varepsilon],$ 
then \eqref{eq:f_increasing} would imply $
\mathcal L_\alpha(R_\varepsilon)\ge N-2-\varepsilon$, a contradiction. Hence for every \(\alpha\in\widetilde{Z}\), there exists
\(\tau_\alpha\in[\rho_\alpha,R_\varepsilon]\) such that $
\mathcal L_\alpha(\tau_\alpha)=N-2-\varepsilon.$
Moreover, if \(\mathcal L_\alpha(\rho)=N-2-\varepsilon\), then 
\eqref{eq:compact_family_f_ineq} gives $
\mathcal L_\alpha'(\rho)>
\frac{\mu_{\theta}\varepsilon}{\rho}\mathcal L_\alpha(\rho)>0.$
Therefore \(\mathcal L_\alpha\) cannot cross the level \(N-2-\varepsilon\) from above to
below. Since it reaches that level no later than \(R_\varepsilon\), we conclude that \eqref{eq:compact_family_f_eventually} holds.

Integrating \eqref{eq:compact_family_f_eventually}, we obtain, for every $\rho\ge R_\varepsilon$,
\[
v_\alpha(\rho)\le v_\alpha(R_\varepsilon)
\Bigl(\frac{\rho}{R_\varepsilon}\Bigr)^{-(N-2-\varepsilon)}
\le
b\,R_\varepsilon^{\,N-2-\varepsilon}\rho^{-(N-2-\varepsilon)}.
\]
Thus we have proved the uniform decay
\begin{equation}\label{eq:compact_family_preliminary_decay}
v_\alpha(\rho)\le C_\varepsilon \rho^{-(N-2-\varepsilon)}
\quad\text{for any }\rho\ge R_\varepsilon,\ \alpha\in\widetilde{Z},
\end{equation}
for some constant \(C_\varepsilon>0\) independent of \(\alpha\).

By arguments similar to the proofs of \eqref{eq:v'_decay} and \eqref{eq:v_decay}, there exists \(C_*>0\), independent of \(\alpha\in\widetilde{Z}\), such that
\[
0<y_\alpha(\rho)\le C_*
\quad\text{for any }\rho\ge R_\varepsilon,\ \alpha\in\widetilde{Z},
\]
and
\[
v_\alpha(\rho)
\le
C_*\rho^{2-N}
\quad\text{for any }\rho\ge R_\varepsilon,\ \alpha\in\widetilde{Z}.
\]
The same estimate for \(u_\alpha\) follows from \(u_\alpha=r(v_\alpha)\le v_\alpha\). Taking \(R_*:=R_\varepsilon\), we obtain
\eqref{eq:compact_family_uniform_tail_v}.

Now \eqref{eq:compact_family_uniform_tail_u} follows immediately. Indeed, for every \(R\ge R_*\) and \(\alpha\in\widetilde Z\),
\[
\begin{aligned}
\int_{|x|>R}u_\alpha(x)^2\,dx
&=|\mathbb S^{N-1}|\int_R^\infty u_\alpha(\rho)^2\,\rho^{N-1}\,d\rho\\
&\le C_*^2 |\mathbb S^{N-1}|\int_R^\infty \rho^{-2N+4}\rho^{N-1}\,d\rho
= C\int_R^\infty \rho^{-N+3}\,d\rho.
\end{aligned}
\]
Since $N\ge 5$, the last integral tends to $0$ as $R\rightarrow\infty$, uniformly in $\alpha\in\widetilde Z$.

Arguing as in the proof of \eqref{eq:sec5_solution_map_L2}, we infer from Lemma~\ref{lem:sec5_continuity} and
\eqref{eq:compact_family_uniform_tail_u} that 
\(\widetilde U\) is continuous. According to the definition of \(\widetilde M\), \(\widetilde M\) is continuous.
\end{proof}

\begin{proof}[Proof of Theorem~\ref{thm:1.3} \textup{$(ii)$}]
By Proposition~\ref{prop:Z_tilde_closed}, \(\widetilde Z\) is compact. By
Proposition~\ref{prop:compact_family_tail_mass},
\(\widetilde U:\widetilde Z\to L^2(\mathbb R^N)\) is continuous. Since
\(\widetilde{\mathcal S}=\widetilde U(\widetilde Z)\),
\(\widetilde{\mathcal S}\) is a nonempty compact subset of
\(L^2(\mathbb R^N)\). Consequently,
\(\widetilde Y=\widetilde M(\widetilde Z)\) is a nonempty compact subset of
\((0,\infty)\).
\end{proof}

\appendix

\section{Existence of a least energy positive solution in Theorem~\ref{thm:1.5}}
\label{app:gen-existence}

In this appendix, we prove the existence part of Theorem~\ref{thm:1.5}.
\subsection{Mountain-pass geometry}

For \(v\in D^{1,2}(\R^N)\), define
\[
J(v)=\frac12\int_{\R^N}|\nabla v|^2\,dx-
\int_{\R^N}K(v)\,dx.
\]

By \textup{(G1)} and \eqref{eq:gen_k_global_upper}, we have
\begin{equation}\label{eq:J_basic_growth}
0\le K(t)\le C(t_+)^{2^*},
\quad
0\le k(t)\le C(t_+)^{2^*-1}, \quad \text{for }t\in \mathbb{R}.
\end{equation}
Thus \(J\in C^1(D^{1,2}(\R^N),\R)\). 

Define
\[
\Gamma=
\left\{\gamma\in C([0,1],D^{1,2}(\R^N)):
\gamma(0)=0,\ J(\gamma(1))<0\right\},
\]
and
\[
\Gamma^r=
\left\{\gamma\in C([0,1],D_{\mathrm{rad}}^{1,2}(\R^N)):
\gamma(0)=0,\ J(\gamma(1))<0\right\}.
\]
Set
\begin{equation}\label{eq:E_def}
E=
\inf_{\gamma\in\Gamma}\ \sup_{t\in[0,1]}J(\gamma(t)).
\end{equation}

\begin{lemma}\label{lem:J_MP_geometry}
Assume that \textup{(G1)}--\textup{(G3)} hold. Then
\begin{equation}\label{eq:J_radial_minimax}
E=
\inf_{\gamma\in\Gamma^r}\sup_{t\in[0,1]}J(\gamma(t)).
\end{equation}
Moreover, there exist \(d,\eta>0\) and \(v_1\in D_{\mathrm{rad}}^{1,2}(\R^N)\) such that
\[
J(v)\ge\eta\quad\text{for }\|\nabla v\|_2=d,
\quad
E\ge\eta,
\quad
J(v_1)<0.
\]
\end{lemma}

\begin{proof}
Since \(\Gamma^r\subset\Gamma\),
\begin{equation}\label{eq:J_radial_minimax_first_ineq}
E\le
\inf_{\gamma\in\Gamma^r}\sup_{t\in[0,1]}J(\gamma(t)).
\end{equation}
For \(\gamma\in\Gamma\), let \(\widetilde\gamma(t)=|\gamma(t)|^*\) be the
Schwarz symmetric decreasing rearrangement of \(|\gamma(t)|\), which implies
\(\widetilde\gamma\in C([0,1],D_{\mathrm{rad}}^{1,2}(\R^N))\). Thus $
\int_{\R^N}|\nabla\widetilde\gamma(t)|^2\,dx
\le
\int_{\R^N}|\nabla \gamma(t)|^2\,dx.$ 
Moreover,
\eqref{eq:J_basic_growth} yields
\[
\int_{\R^N}K(\widetilde\gamma(t))\,dx
=\int_{\R^N}K(|\gamma(t)|)\,dx
\ge
\int_{\R^N}K(\gamma(t))\,dx.
\]
Thus \(J(\widetilde\gamma(t))\le J(\gamma(t))\) for any
\(t\in[0,1]\) and \(\widetilde\gamma\in\Gamma^r\). Therefore
\[
\inf_{\hat{\gamma} \in\Gamma^r}\sup_{t\in[0,1]}J(\hat{\gamma}(t))
\le \sup_{t\in[0,1]}J(\gamma(t)).
\]
Since \(\gamma\in\Gamma\) is arbitrary,
\[
\inf_{\hat{\gamma}\in\Gamma^r}\sup_{t\in[0,1]}J(\hat{\gamma}(t))
\le
\inf_{\gamma\in\Gamma}\sup_{t\in[0,1]}J(\gamma(t))
=E.
\]
This together with \eqref{eq:J_radial_minimax_first_ineq} yields
\eqref{eq:J_radial_minimax}.

By \eqref{eq:J_basic_growth} and the Sobolev inequality, we obtain 
\[
J(v)\ge\frac12\|\nabla v\|_2^2-C\|\nabla v\|_2^{2^*}, 
\]
which implies that there exist \(d,\eta>0\) such that
\begin{equation}\label{eq:J_MP_geometry_local_bound}
J(v)\ge0\quad\text{for }\|\nabla v\|_2\le d,
\quad
J(v)\ge\eta\quad\text{for }\|\nabla v\|_2=d.
\end{equation}

Let \(0\ne\varphi\in C_c^\infty(\R^N)\cap D_{\mathrm{rad}}^{1,2}(\R^N)\), \(\varphi\ge0\).
By Lemma~\ref{lem:prelim_gen_dual_decomposition} \textup{$(i)$, $(iii)$} and
\(K_{\mathrm p}\ge0\), there exist \(c,C>0\) such that, for all \(t\ge0\),
\(K(t)\ge c t^{2^*}-C\). Hence \(J(s\varphi)\to-\infty\) as
\(s\to+\infty\). Taking \(s_0>0\) sufficiently large and setting
\(v_1=s_0\varphi\), we obtain \(J(v_1)<0\).

Let \(\gamma\in\Gamma\). By \eqref{eq:J_MP_geometry_local_bound} and
\(J(\gamma(1))<0\), the continuity of \(\gamma\) yields a point \(t_0\in(0,1)\)
such that \(\|\nabla\gamma(t_0)\|_2=d\). Thus
\(\sup_{t\in[0,1]}J(\gamma(t))\ge\eta\), and the definition of \(E\) yields
\(E\ge\eta\).
\end{proof}

\begin{proposition}\label{prop:J_below_cstar}
Assume that \textup{(G1)}--\textup{(G3)} hold. Then \(E<\mathcal E_*\).
\end{proposition}

\begin{proof}
We first claim that there exist \(c_0,C_0>0\) such that
\begin{equation}\label{eq:J_below_cstar_global_K_lower}
K(s)
\ge
\frac{2^{2/(N-2)}}{2^*}s^{2^*}
+c_0s^{q_1/2}
-C_0s^{2^*-1}\ln(1+s)
-C_0s^{2^*-1}
\quad\forall s\ge0.
\end{equation}
Indeed, Lemma~\ref{lem:prelim_gen_dual_decomposition} \textup{$(ii)$, $(iii)$}
imply that there exists \(R>0\) such that \eqref{eq:J_below_cstar_global_K_lower} holds for \(s\ge R\). Since \(2^*-1<q_1/2<2^*\), increasing \(C_0\), if
necessary, yields
\[
\frac{2^{2/(N-2)}}{2^*}s^{2^*}+c_0s^{q_1/2}
\le C_0s^{2^*-1}(1+\ln(1+s))
\quad\text{for }0\le s\le R.
\]
Noting that \(K\ge0\), we obtain that 
\eqref{eq:J_below_cstar_global_K_lower} holds on \([0,R]\).

Let
\[
U_\varepsilon(x)=\alpha_N
\left(\frac{\varepsilon}{\varepsilon^2+|x|^2}\right)^{\frac{N-2}{2}},
\quad
\psi_\varepsilon=\zeta U_\varepsilon,
\]
where \(\alpha_N=[N(N-2)]^{(N-2)/4}\) and
\(\zeta\in C_c^\infty(\R^N)\) is radial with \(0\le\zeta\le1\),
\(\zeta\equiv1\) in \(B_1(0)\), and
\(\operatorname{supp}\zeta\subset B_2(0)\). Set
\[
A_\varepsilon=\int_{\R^N}|\nabla\psi_\varepsilon|^2\,dx,
\quad
B_\varepsilon=\int_{\R^N}\psi_\varepsilon^{2^*}\,dx,
\quad
D_\varepsilon=\int_{\R^N}\psi_\varepsilon^{q_1/2}\,dx.
\]
The standard estimates for cut-off Aubin--Talenti functions
\cite[Lemma~1.46]{Willem1996} and
\cite[Appendix, Lemma~7.1]{MR4476243} yield, as
\(\varepsilon\to0^+\),
\begin{equation}\label{eq:J_below_cstar_bubble_estimates}
\begin{gathered}
A_\varepsilon=S^{N/2}+O(\varepsilon^{N-2}),
\quad
B_\varepsilon=S^{N/2}+O(\varepsilon^N),\\
c\varepsilon^{N-\frac{N-2}{4}q_1}
\le D_\varepsilon
\le C\varepsilon^{N-\frac{N-2}{4}q_1},
\quad
\int_{\R^N}\psi_\varepsilon^{2^*-1}\,dx
=O\bigl(\varepsilon^{\frac{N-2}{2}}\bigr).
\end{gathered}
\end{equation}
By \textup{(G3)},
\begin{equation}\label{eq:J_below_cstar_exponent_range}
0<N-\frac{N-2}{4}q_1<\frac{N-2}{2}.
\end{equation}
Since \(\|\psi_\varepsilon\|_{L^\infty}\le C\varepsilon^{-(N-2)/2}\),
\begin{equation}\label{eq:J_below_cstar_log}
\int_{\R^N}\psi_\varepsilon^{2^*-1}\ln(1+\psi_\varepsilon)\,dx
=O\left(\varepsilon^{\frac{N-2}{2}}\ln\frac1\varepsilon\right)
\quad\text{as }\varepsilon\to0^+.
\end{equation}

For \(t\ge0\), \eqref{eq:J_below_cstar_global_K_lower} yields
\begin{align}
J(t\psi_\varepsilon)
&\le
\frac{t^2}{2}A_\varepsilon
-
\frac{2^{2/(N-2)}t^{2^*}}{2^*}B_\varepsilon
-c_0t^{q_1/2}D_\varepsilon
\notag\\
&\quad+
C_0t^{2^*-1}\int_{\R^N}\psi_\varepsilon^{2^*-1}\ln(1+t\psi_\varepsilon)\,dx
+C_0t^{2^*-1}\int_{\R^N}\psi_\varepsilon^{2^*-1}\,dx.
\label{eq:J_below_cstar_main_upper}
\end{align}
The estimates in \eqref{eq:J_below_cstar_bubble_estimates} imply
\begin{equation}\label{eq:J_below_cstar_Gmax}
\sup_{t\ge0}
\left\{
\frac{A_\varepsilon}{2}t^2
-\frac{2^{2/(N-2)}B_\varepsilon}{2^*}t^{2^*}
\right\}
=\mathcal E_*+O(\varepsilon^{N-2}).
\end{equation}
Moreover, \eqref{eq:J_below_cstar_global_K_lower} and
\(s^{2^*-1}(1+\ln(1+s))=o(s^{2^*})\) as \(s\to+\infty\) yield constants
\(a_1,C_1>0\) such that \(K(s)\ge a_1s^{2^*}-C_1\) for all \(s\ge0\).
Since \(\operatorname{supp}\psi_\varepsilon\subset B_2(0)\),
\[
J(t\psi_\varepsilon)
\le
\frac{t^2}{2}A_\varepsilon-a_1t^{2^*}B_\varepsilon+C_1|B_2(0)|.
\]
By \eqref{eq:J_below_cstar_bubble_estimates}, there exist
\(\varepsilon_0,T>0\) such that, for \(0<\varepsilon<\varepsilon_0\),
\begin{equation}\label{eq:J_below_cstar_tail_negative}
J(t\psi_\varepsilon)<0
\quad\text{for }t\ge T.
\end{equation}
Since \(J(0)=0\), \eqref{eq:J_below_cstar_tail_negative} implies
that, for \(0<\varepsilon<\varepsilon_0\),
\begin{equation}\label{eq:J_below_cstar_sup_reduction}
\sup_{t\ge0}J(t\psi_\varepsilon)
=
\sup_{0\le t\le T}J(t\psi_\varepsilon).
\end{equation}
By \eqref{eq:J_below_cstar_bubble_estimates} and
\eqref{eq:J_below_cstar_log},
\begin{equation}\label{eq:J_below_cstar_log_T}
\sup_{0\le t\le T}
t^{2^*-1}\int_{\R^N}\psi_\varepsilon^{2^*-1}
\ln(1+t\psi_\varepsilon)\,dx
=
O\left(\varepsilon^{\frac{N-2}{2}}\ln\frac1\varepsilon\right)
\quad\text{as }\varepsilon\to0^+.
\end{equation}
It follows from \eqref{eq:J_below_cstar_bubble_estimates}, \eqref{eq:J_below_cstar_exponent_range}, \eqref{eq:J_below_cstar_main_upper}, \eqref{eq:J_below_cstar_Gmax},
\eqref{eq:J_below_cstar_sup_reduction} and \eqref{eq:J_below_cstar_log_T} that there exists
\(\varepsilon_1\in(0,\varepsilon_0)\) such that, for
\(0<\varepsilon<\varepsilon_1\),
\[
\sup_{t\ge0}J(t\psi_\varepsilon)
=
\sup_{0\le t\le T}J(t\psi_\varepsilon)
\le
\mathcal E_*-c\varepsilon^{N-\frac{N-2}{4}q_1}
<\mathcal E_*.
\]
For \(0<\varepsilon<\varepsilon_1\), define
\[
\gamma_\varepsilon:[0,1]\to D_{\mathrm{rad}}^{1,2}(\R^N),
\quad
\gamma_\varepsilon(s)=sT\psi_\varepsilon.
\]
Since \(\gamma_\varepsilon(0)=0\) and
\eqref{eq:J_below_cstar_tail_negative} holds at \(t=T\),
\(\gamma_\varepsilon\in\Gamma^r\). Hence
\eqref{eq:J_radial_minimax} yields \(E<\mathcal E_*\).
\end{proof}

\subsection[Palais--Smale compactness below E*]{Palais--Smale compactness below \texorpdfstring{$\mathcal E_*$}{E*}}

\begin{proposition}\label{prop:J_PS_below_cstar}
Assume that \textup{(G1)}--\textup{(G3)} hold. Let
\(\{v_n\}\subset D_{\mathrm{rad}}^{1,2}(\R^N)\) be a bounded Palais--Smale sequence for
\(J|_{D_{\mathrm{rad}}^{1,2}(\R^N)}\) with \(J(v_n)\to c_{\mathrm{ps}}\in(0,\mathcal E_*)\). 
Then there exists a
nontrivial nonnegative weak solution \(v\in D_{\mathrm{rad}}^{1,2}(\R^N)\) of
\eqref{eq:intro_dual} such that \(v_n\to v\) in $D_{\mathrm{rad}}^{1,2}(\R^N)$.
\end{proposition}

\begin{proof}
Fix \(R_c>1\), and choose a nondecreasing
\(\eta\in C^\infty([0,\infty))\), \(0\le\eta\le1\), such that
\(\eta=0\) on \([0,R_c]\) and
\(\eta=1\) on \([2R_c,\infty)\). For \(s\in\R\), define
\(f_c(s)=2^{2/(N-2)}\eta(s_+)(s_+)^{2^*-1}\) and
\(F_c(s)=\int_0^s f_c(\tau)\,d\tau\). Let
\(b_c=k-f_c\) and \(B_c=K-F_c\).
Thus
\[
J(v)=\frac12\|\nabla v\|_2^2-\int_{\R^N}F_c(v)\,dx
-\int_{\R^N}B_c(v)\,dx.
\]
Lemmas~\ref{lem:prelim_k_local_bounds} and
\ref{lem:prelim_gen_dual_decomposition} \textup{$(i)$--$(iii)$}, together with the
definition of \(b_c\), yield
\begin{equation}\label{eq:PS_bc_subcritical}
\lim\limits_{s\to0}\frac{|B_c(s)|}{|s|^{2^*}}
=\lim\limits_{|s|\to\infty}\frac{|B_c(s)|}{|s|^{2^*}}
=\lim\limits_{s\to0}\frac{|b_c(s)|}{|s|^{2^*-1}}
=\lim\limits_{|s|\to\infty}\frac{|b_c(s)|}{|s|^{2^*-1}}=0.
\end{equation}

By the boundedness of \(\{v_n\}\) in \(D_{\mathrm{rad}}^{1,2}(\R^N)\), after passing to a
subsequence there exists \(v\in D_{\mathrm{rad}}^{1,2}(\R^N)\) such that
\begin{equation}\label{eq:PS_vn_weak_ae}
v_n\rightharpoonup v\quad\text{in }D_{\mathrm{rad}}^{1,2}(\R^N),\quad
v_n\to v\quad\text{a.e. in }\R^N.
\end{equation}
By \eqref{eq:PS_bc_subcritical} and \cite[Lemma~A.1 (b)]{MR4215019}, we obtain 
\begin{equation}\label{eq:PS_Bc_compact}
\int_{\R^N}B_c(v_n)\,dx\to \int_{\R^N}B_c(v)\,dx,
\quad
\int_{\R^N}b_c(v_n)v_n\,dx
\to
\int_{\R^N}b_c(v)v\,dx.
\end{equation}
Moreover, \eqref{eq:PS_bc_subcritical} and the definition of \(f_c\) imply that
there exists \(C>0\) such that
\begin{equation}\label{eq:PS_bc_fc_growth}
|b_c(s)|+|f_c(s)|\le C|s|^{2^*-1}
\quad\text{for all }s\in\R.
\end{equation}
By \eqref{eq:PS_bc_fc_growth}, \eqref{eq:PS_vn_weak_ae}, and the Sobolev embedding,
\[
b_c(v_n)\rightharpoonup b_c(v),\quad
f_c(v_n)\rightharpoonup f_c(v)
\quad\text{in }L^{2N/(N+2)}(\R^N).
\]
Thus, for any \(\varphi\in D_{\mathrm{rad}}^{1,2}(\R^N)\),
\begin{equation}\label{eq:PS_bc_test}
\int_{\R^N} b_c(v_n)\varphi\,dx
\to
\int_{\R^N} b_c(v)\varphi\,dx,
\end{equation}
and
\begin{equation}\label{eq:PS_fc_test}
\int_{\R^N}f_c(v_n)\varphi\,dx\to
\int_{\R^N}f_c(v)\varphi\,dx.
\end{equation}
Let \(w_n=v_n-v\).
Moreover, it follows from \eqref{eq:PS_Bc_compact} and \eqref{eq:PS_bc_test} that
\begin{equation}\label{eq:PS_bc_wn}
\int_{\R^N}\bigl(b_c(v_n)-b_c(v)\bigr)w_n\,dx=o(1).
\end{equation}
By \eqref{eq:PS_vn_weak_ae}, \eqref{eq:PS_bc_test}, \eqref{eq:PS_fc_test}, and
$J^{\prime}(v_n) \to 0$ in $\bigl(D_{\mathrm{rad}}^{1,2}(\R^N)\bigr)^*$, for any \(\varphi\in D_{\mathrm{rad}}^{1,2}(\R^N)\),
\[
\langle J'(v),\varphi\rangle
=\lim\limits_{n\to\infty}\langle J'(v_n),\varphi\rangle=0.
\]

The definition of \(f_c\) yields
\begin{equation}\label{eq:PS_fc_growth}
|F_c(s)|\le C|s|^{2^*},\quad
|f_c(s)|\le C|s|^{2^*-1},\quad
|f_c'(s)|\le C|s|^{2^*-2},
\end{equation}
and \cite[Lemma~A.1 (a)]{MR4215019} yields
\begin{equation}\label{eq:PS_Fc_split}
\int_{\R^N}F_c(v_n)\,dx
=\int_{\R^N}F_c(v)\,dx+\int_{\R^N}F_c(w_n)\,dx+o(1).
\end{equation}
Fix \(\varepsilon>0\). Let \(M\ge2\) and \(a,b\in\R\). If \(|b|\le M|a|\), then
\eqref{eq:PS_fc_growth} gives
\begin{equation}\label{eq:PS_fc_split_case_small}
|f_c(a+b)-f_c(a)-f_c(b)|
\le C_M |a|^{2^*-1}.
\end{equation}
If \(|b|>M|a|\), then
\begin{equation}\label{eq:PS_fc_split_case_large}
\begin{aligned}
|f_c(a+b)-f_c(a)-f_c(b)|
&\le |f_c(a+b)-f_c(b)|+|f_c(a)|\\
&=\left|\int_0^1 f_c'(b+ta)a\,dt\right|+|f_c(a)|\\
&\le C(1+M^{-1})^{2^*-2}|a||b|^{2^*-2}
   +C|a|^{2^*-1}\\
&\le C(1+M^{-1})^{2^*-2}M^{-1}|b|^{2^*-1}
   +C|a|^{2^*-1}.
\end{aligned}
\end{equation}
Taking \(M\) sufficiently large in \eqref{eq:PS_fc_split_case_large} and using
\eqref{eq:PS_fc_split_case_small}, we obtain, for some \(C_\varepsilon>0\),
\begin{equation}
|f_c(s+t)-f_c(a)-f_c(t)|
\le \varepsilon |t|^{2^*-1}+C_\varepsilon |s|^{2^*-1}
\end{equation}
for all \(s,t\in\R\), which implies that 
\[
\Bigl(
|f_c(v_n)-f_c(v)-f_c(w_n)|-\varepsilon |w_n|^{2^*-1}
\Bigr)_+^{2N/(N+2)}
\le C_\varepsilon |v|^{2^*}.
\]
By \eqref{eq:PS_vn_weak_ae} and the
dominated convergence theorem, we obtain 
\begin{equation}\label{eq:PS_fc_positive_part}
\Bigl(
|f_c(v_n)-f_c(v)-f_c(w_n)|-\varepsilon |w_n|^{2^*-1}
\Bigr)_+
\to0
\quad\text{in }L^{2N/(N+2)}(\R^N).
\end{equation}
By \eqref{eq:PS_fc_positive_part} and the boundedness of \(\{w_n\}\) in
\(L^{2^*}(\R^N)\),
\[
\limsup\limits_{n\to\infty}
\|f_c(v_n)-f_c(v)-f_c(w_n)\|_{L^{2N/(N+2)}(\R^N)}
\le C\varepsilon,
\]
which implies that 
\begin{equation}\label{eq:PS_fc_split}
f_c(v_n)-f_c(v)-f_c(w_n)
\to0
\quad\text{in }\bigl(D^{1,2}(\R^N)\bigr)^*.
\end{equation}

Combining \eqref{eq:PS_Bc_compact},
\eqref{eq:PS_Fc_split}, and
\[
\|\nabla v_n\|_2^2=\|\nabla v\|_2^2+\|\nabla w_n\|_2^2+o(1),
\]
we obtain 
\begin{equation}\label{eq:PS_split_energy_fixed}
J(v_n)=J(v)+\frac12\|\nabla w_n\|_2^2
-\int_{\R^N}F_c(w_n)\,dx+o(1).
\end{equation}
Since \(J'(v)=0\), \eqref{eq:PS_bc_wn} and \eqref{eq:PS_fc_split} yield
\begin{equation}\label{eq:PS_split_derivative_fixed}
\langle J'(v_n),w_n\rangle
=\|\nabla w_n\|_2^2-\int_{\R^N}f_c(w_n)w_n\,dx+o(1).
\end{equation}

Suppose by contradiction that $
\|\nabla w_n\|_2^2\to L_w>0.$ 
Since \(J'(v_n)\to0\) in \(\bigl(D_{\mathrm{rad}}^{1,2}(\R^N)\bigr)^*\) and
\(\{w_n\}\) is bounded in \(D_{\mathrm{rad}}^{1,2}(\R^N)\),
\eqref{eq:PS_split_derivative_fixed} yields
\begin{equation}
L_w=\lim\limits_{n\to\infty}\int_{\R^N}f_c(w_n)w_n\,dx
=2^{2/(N-2)}
\lim\limits_{n\to\infty}\int_{\R^N}\eta((w_n)_+)((w_n)_+)^{2^*}\,dx.
\end{equation}
By \(0\le\eta\le1\) and the Sobolev inequality,
\begin{equation}\label{eq:PS_Lw_sobolev_lower}
2^{-2/(N-2)}L_w
\le
\limsup\limits_{n\to\infty}\int_{\R^N}((w_n)_+)^{2^*}\,dx
\le
S^{-2^*/2}L_w^{2^*/2},
\end{equation}
hence $
L_w\ge\frac12S^{N/2}.$

Since \(\eta\) is nondecreasing,
\begin{equation}\label{eq:PS_Fc_fc_bound}
F_c(s)=2^{2/(N-2)}
\int_0^{s_+}\eta(\tau)\tau^{2^*-1}\,d\tau
\le \frac1{2^*}f_c(s)s
\quad\text{for }s\in\R.
\end{equation}
Moreover, \(J'(v_n)\to0\) in \(\bigl(D_{\mathrm{rad}}^{1,2}(\R^N)\bigr)^*\), the boundedness of
\(\{w_n\}\), and \eqref{eq:PS_split_derivative_fixed} imply
\begin{equation}\label{eq:PS_wn_nehari_defect}
\|\nabla w_n\|_2^2-\int_{\R^N}f_c(w_n)w_n\,dx=o(1).
\end{equation}
Using \(J(v_n)\to c_{\mathrm{ps}}\), \eqref{eq:PS_split_energy_fixed},
\eqref{eq:PS_Fc_fc_bound}, and \eqref{eq:PS_wn_nehari_defect}, we obtain
\begin{equation}\label{eq:PS_energy_liminf_bound}
\begin{aligned}
c_{\mathrm{ps}}
&=J(v)+
\lim\limits_{n\to\infty}\left[\frac12\|\nabla w_n\|_2^2-
\int_{\R^N}F_c(w_n)\,dx\right]\\
&\ge
J(v)+\liminf\limits_{n\to\infty}
\left[\frac12\|\nabla w_n\|_2^2-\frac1{2^*}
\int_{\R^N}f_c(w_n)w_n\,dx\right]\\
&=
J(v)+\lim\limits_{n\to\infty}
\left[\left(\frac12-\frac1{2^*}\right)\|\nabla w_n\|_2^2
\right]\\
&=
J(v)+\left(\frac12-\frac1{2^*}\right)L_w
=J(v)+\frac1N L_w.
\end{aligned}
\end{equation}
The Pohozaev identity 
yields $
\frac{N-2}{2}\|\nabla v\|_2^2=N\int_{\R^N}K(v)\,dx,$
and hence \(J(v)=N^{-1}\|\nabla v\|_2^2\ge0\). Therefore
\[
c_{\mathrm{ps}}\ge\frac1N L_w\ge\frac1{2N}S^{N/2}=\mathcal E_*,
\]
by \eqref{eq:PS_energy_liminf_bound} and
\eqref{eq:PS_Lw_sobolev_lower}, contradicting \(c_{\mathrm{ps}}<\mathcal E_*\).

Hence \(w_n\to0\) in
\(D_{\mathrm{rad}}^{1,2}(\R^N)\). By \eqref{eq:PS_split_energy_fixed}, we have \(J(v)=c_{\mathrm{ps}}\). Moreover, by \eqref{eq:J_basic_growth},
\begin{equation}\label{eq:PS_nonnegative_test}
0=\langle J'(v),v_-\rangle
=-\int_{\R^N}|\nabla v_-|^2\,dx.
\end{equation}
Therefore \(v\ge0\). We complete the proof.
\end{proof}

\subsection{Existence at the least energy level}

\begin{proposition}\label{prop:sec3_groundstate_identification}
Assume that \textup{(G1)}--\textup{(G3)} hold. Then \(J\) admits a positive radial critical point \(v\in D_{\mathrm{rad}}^{1,2}(\R^N)\)
with \(J(v)=E\). Moreover,
\begin{equation}\label{eq:ground_state_level_identity}
E=
\inf\left\{
J(w):\ w\in D^{1,2}(\R^N)\setminus\{0\},\ J'(w)=0
\right\}.
\end{equation}
\end{proposition}

\begin{proof}
Define \(\Phi:\R\times D_{\mathrm{rad}}^{1,2}(\R^N)\to\R\) by
\[
\Phi(\theta,v)=J(v(e^{-\theta}\cdot))
=\frac{e^{(N-2)\theta}}2\|\nabla v\|_2^2
-e^{N\theta}\int_{\R^N}K(v)\,dx.
\]
Since \(J\in C^1(D^{1,2}(\R^N),\R)\), $
\Phi\in C^1(\R\times D_{\mathrm{rad}}^{1,2}(\R^N),\R).$ Define
\[
\widehat\Gamma=
\left\{
\eta\in C([0,1],\R\times D_{\mathrm{rad}}^{1,2}(\R^N)):
\eta(0)=(0,0),\ \Phi(\eta(1))<0
\right\}.
\]
For \(\eta\in\widehat\Gamma\), write \(\eta(t)=(\theta(t),v(t))\) and define $
\gamma_\eta(t)=v(t)(e^{-\theta(t)}\cdot)$ for $ t\in[0,1].$ 
Then \(\gamma_\eta\in\Gamma^r\) and
\(J(\gamma_\eta(t))=\Phi(\eta(t))\) for \(t\in[0,1]\). Conversely, if
\(\gamma\in\Gamma^r\), then \(\eta_\gamma(t)=(0,\gamma(t))\) belongs to
\(\widehat\Gamma\) and \(\Phi(\eta_\gamma(t))=J(\gamma(t))\). Hence, by
\eqref{eq:J_radial_minimax},
\[
E=\inf_{\eta\in\widehat\Gamma}\max_{t\in[0,1]}\Phi(\eta(t)).
\]

Let \(\varepsilon_n\to0^+\) be such that \(2\varepsilon_n^2<E\) for any
\(n\). By \eqref{eq:J_radial_minimax}, choose \(\gamma_n\in\Gamma^r\) such that
\[
\max_{t\in[0,1]}J(\gamma_n(t))\le E+\varepsilon_n^2.
\]
Set \(\eta_n(t)=(0,\gamma_n(t))\) for \(t\in[0,1]\).
By Lemma~\ref{lem:J_MP_geometry}, \(E>0\). Hence,
\[
\Phi(\eta_n(0))=0<E,\quad
\Phi(\eta_n(1))=J(\gamma_n(1))<0<E.
\]
By \cite[Theorem~2.8]{Willem1996}, we obtain that there exists
\(z_n=(\theta_n,\xi_n)\in\R\times D_{\mathrm{rad}}^{1,2}(\R^N)\) such that
\[
\operatorname{dist}_{\R\times D_{\mathrm{rad}}^{1,2}(\R^N)}
\bigl(z_n,\eta_n([0,1])\bigr)\le2\varepsilon_n,
\quad
|\Phi(z_n)-E|\le2\varepsilon_n^2,
\quad
\|D\Phi(z_n)\|_{(\R\times D_{\mathrm{rad}}^{1,2}(\R^N))^*}\le8\varepsilon_n.
\]

Thus, $
|\theta_n|\le
\operatorname{dist}_{\R\times D_{\mathrm{rad}}^{1,2}(\R^N)}
\bigl(z_n,\eta_n([0,1])\bigr)
\le2\varepsilon_n,$ which implies \(\theta_n\to0\). Define
\(v_n(x)=\xi_n(e^{-\theta_n}x)\). So $\Phi(z_n)=J(v_n) \to E.$ Noting that, for any
\(\psi\in D_{\mathrm{rad}}^{1,2}(\R^N)\),
\[
\langle (J|_{D_{\mathrm{rad}}^{1,2}(\R^N)})'(v_n),\psi\rangle
=\partial_v\Phi(\theta_n,\xi_n)[\psi(e^{\theta_n}\cdot)]
\]
and $
\|\nabla(\psi(e^{\theta_n}\cdot))\|_2
=e^{-(N-2)\theta_n/2}\|\nabla\psi\|_2.$
Therefore, by \(\|D\Phi(z_n)\|_{(\R\times D_{\mathrm{rad}}^{1,2}(\R^N))^*}\le8\varepsilon_n\), we obtain that $$
\left|\langle J'(v_n),\psi\rangle\right|
\le
8\varepsilon_n e^{-(N-2)\theta_n/2}\|\nabla\psi\|_2,$$
and therefore $
J'(v_n)\to0$ in $\bigl(D_{\mathrm{rad}}^{1,2}(\R^N)\bigr)^*.$

For \(v\in D^{1,2}(\R^N)\), define
\[
P(v)=\frac{N-2}{2}\int_{\R^N}|\nabla v|^2\,dx
-N\int_{\R^N}K(v)\,dx.
\]
Then
\[
\partial_\theta\Phi(\theta_n,\xi_n)
=\frac{N-2}{2}e^{(N-2)\theta_n}\|\nabla\xi_n\|_2^2
-N e^{N\theta_n}\int_{\R^N}K(\xi_n)\,dx
=P(v_n).
\]
Since
\[
|\partial_\theta\Phi(\theta_n,\xi_n)|
\le\|D\Phi(z_n)\|_{(\R\times D_{\mathrm{rad}}^{1,2}(\R^N))^*}
\le8\varepsilon_n,
\]
we have \(P(v_n)\to0\). Therefore
\[
\|\nabla v_n\|_2^2=NJ(v_n)-P(v_n)=NE+o(1),
\]
and hence \(\{v_n\}\) is bounded in \(D_{\mathrm{rad}}^{1,2}(\R^N)\).

It follows from Lemma~\ref{lem:J_MP_geometry}
and Proposition~\ref{prop:J_below_cstar} that
\(E\in(0,\mathcal E_*)\). Hence Proposition~\ref{prop:J_PS_below_cstar} yields,
after passing to a subsequence, a nontrivial nonnegative \(v\in D_{\mathrm{rad}}^{1,2}(\R^N)\) such that
\[
v_n\to v\quad\text{in }D_{\mathrm{rad}}^{1,2}(\R^N),
\quad
J'(v)=0,\quad J(v)=E.
\]
The strong maximum principle yields \(v>0\) in \(\R^N\). Therefore
\begin{equation}\label{eq:ground_state_level_upper}
\inf\left\{
J(w):\ w\in D^{1,2}(\R^N)\setminus\{0\},\ J'(w)=0
\right\}
\le E.
\end{equation}

Conversely, let \(w\in D^{1,2}(\R^N)\setminus\{0\}\) satisfy \(J'(w)=0\). As in \eqref{eq:PS_nonnegative_test}, we have \(w\ge0\). For
\(t>0\), set \(w_t(x)=w(x/t)\). The Pohozaev identity gives
\[
\int_{\R^N}K(w)\,dx=\frac{N-2}{2N}\|\nabla w\|_2^2>0.
\]
Hence
\[
J(w_t)=\frac{t^{N-2}}{2}\int_{\R^N}|\nabla w|^2\,dx
-t^N\int_{\R^N}K(w)\,dx
\to0
\quad\text{as }t\to0^+,
\]
and \(J(w_t)\to-\infty\) as \(t\to+\infty\). Choose \(T>1\) such that
\(J(w_T)<0\), and define \(\gamma(0)=0\), \(\gamma(s)=w_{sT}\) for
\(s\in(0,1]\). Since \(w_t\to0\) in \(D^{1,2}(\R^N)\) as \(t\to0^+\),
\(\gamma\in\Gamma\). Moreover, the Pohozaev identity gives \(P(w)=0\). Hence, for any \(t>0\),
\[
J(w_t)=J(w_t)-\frac{t^N}{N}P(w)
=\left(\frac{t^{N-2}}2-\frac{N-2}{2N}t^N\right)\|\nabla w\|_2^2
\le \frac1N\|\nabla w\|_2^2=J(w).
\]
 By
\eqref{eq:E_def},
\[
E\le\sup_{s\in[0,1]}J(\gamma(s))\le J(w), 
\]
which implies that
\begin{equation}\label{eq:ground_state_level_lower}
E\le
\inf\left\{
J(w):\ w\in D^{1,2}(\R^N)\setminus\{0\},\ J'(w)=0
\right\}.
\end{equation}
By \eqref{eq:ground_state_level_upper} and
\eqref{eq:ground_state_level_lower}, \eqref{eq:ground_state_level_identity}
holds.
\end{proof}

\begin{proposition}
\label{prop:gen_radial_ground_state}
Assume that \(N\ge5\) and that \textup{(G1)}--\textup{(G3)} hold. Then
\eqref{eq:1.1} admits a radially decreasing least energy positive solution
\(u\in X\).
\end{proposition}

\begin{proof}
By Proposition~\ref{prop:sec3_groundstate_identification}, \eqref{eq:intro_dual}
admits a positive radial critical point \(v\in D_{\mathrm{rad}}^{1,2}(\R^N)\) with \(J(v)=E\). By
Remark~\ref{rem:intro_radial_monotonicity}, \(v\) is radially decreasing. In view of \eqref{eq:J_basic_growth}, 
\cite[Theorem~5]{MR1617704} yields $
v(\rho)\le C\rho^{2-N}$ for sufficiently large $\rho.$ Since \(N\ge5\), \(v\in H^1(\R^N)\).

Set \(u=r(v)\). Noting that
$$I(u)=J(v)= \frac12\int_{\R^N}(1+2u^2)|\nabla u|^2\,dx-
\int_{\R^N}G(u)\,dx, $$ 
and using \eqref{eq:ground_state_level_identity}, we conclude that \(u\) is a
least energy positive solution of \eqref{eq:1.1}. Since \(r\) is increasing,
\(u\) is radially
decreasing. We complete the proof.
\end{proof}

\subsection*{Acknowledgments} 
H. Liu is supported by NSFC (12571121, 12171204, 12371107). C. Wu is supported by NSFC (12071266).

\end{document}